\documentclass[12pt]{amsart}

\usepackage{amssymb,mathrsfs,amsmath,amsthm,color,bm,mathtools,bbm,wasysym,subfig,float}
\usepackage[noadjust]{cite}

\usepackage{enumerate}
\usepackage[centering]{geometry}
\allowdisplaybreaks

\theoremstyle{plain}
\newtheorem{thm}{Theorem}[section]
\newtheorem{defn}[thm]{Definition}

\newtheorem{prop}[thm]{Proposition}
\newtheorem{cor}[thm]{Corollary}
\newtheorem{lem}[thm]{Lemma}

\theoremstyle{remark}
\newtheorem{rem}{Remark}
\numberwithin{equation}{section}

\DeclareMathOperator{\hdim}{\dim_H}
\DeclareMathOperator{\fdim}{\dim_F}

\newcommand{\dif}{ \, \mathrm d}
\newcommand{\rc}{\mathcal R}

\newcommand{\supp}{\mathrm{supp}}

\newcommand{\N}{\mathbb N}
\newcommand{\R}{\mathbb R}
\newcommand{\Z}{\mathbb Z}
\newcommand{\lm}{\mathcal L}
\renewcommand{\hm}{\mathcal H}
\newcommand{\hc}{\mathcal H_\infty}

\newcommand{\ck}{\mathcal K}
\newcommand{\cp}{\mathcal P}

\newcommand{\cb}{\mathcal B}

\newcommand{\ic}{\mathrm{i}}

\newcommand{\ca}{\mathcal A}

\newcommand{\cm}{\mathcal M}

\newcommand{\bd}{\bm \delta}
\newcommand{\bx}{\mathbf{x}}

\newcommand{\by}{\mathbf{y}}
\newcommand{\cha}{\mathbf{1}}
\newcommand{\bz}{\mathbf{z}}
\newcommand{\bk}{\mathbf{k}}
\newcommand{\bl}{\mathbf{l}}

\newcommand{\bq}{\mathbf{q}}
\newcommand{\bp}{\mathbf{p}}

\newcommand{\bt}{\mathbf{t}}

\newcommand{\qaq}{\mathrm{\quad and\quad}}

\newcommand{\red}[1]{\textcolor{red}{#1}}

\begin{document}
		\title[Hausdorff measure and Fourier dimensions]{Hausdorff measure and Fourier dimensions of limsup sets arising in weighted and multiplicative Diophantine approximation}
	\author{Yubin He}

	\address{Department of Mathematics, Shantou University, Shantou, Guangdong, 515063, China}

	\email{ybhe@stu.edu.cn}

%
%


	\begin{abstract}
		 The classical Khintchine--Jarn\'ik Theorem provides elegant criteria for determining the Lebesgue measure and Hausdorff measure of sets of points approximated by rational points, which has inspired much modern research in metric Diophantine approximation. This paper concerns the Lebesgue measure, Hausdorff measure and Fourier dimensions of sets arising in weighted and multiplicative Diophantine approximation.

		 We provide zero-full laws for determining the Lebesgue measure and Hausdorff measure of the sets under consideration. In particular, the criterion for the weighted setup refines a dimensional result given by Li, Liao, Velani, Wang, and Zorin [Adv. Math. (2025)], while the criteria for the multiplicative setup answer a question raised by Hussain and Simmons [J. Number Theory (2018)] and extend beyond it.
		 A crucial observation is that, even in higher dimensions, both setups are more appropriately understood as consequences of the `balls-to-rectangles' mass transference principle.

		We also determine the Fourier dimensions of these sets. The results we obtain indicate that, in line with the existence results, these sets are generally non-Salem sets, except in the one-dimensional case. This phenomenon can be partly explained by another result of this paper, which states that the Fourier dimension of the product of two sets equals the minimum of their respective Fourier dimensions.
	\end{abstract}
	\maketitle
\section{Introduction}\label{s:intro}
Let $n\ge 1$ and $m\ge 1$ be integers. Let $\Psi=(\psi_1,\dots,\psi_m)$ be an $m$-tuple of multivariable functions with $\psi_j:\Z^n\to \R_{\ge 0}$ for $1\le j\le m$. Each $\psi_j$ is also referred to as an {\em approximating function}. If $\psi_j(\bq_1)=\psi_j(\bq_2)$ whenever $|\bq_1|=|\bq_2|$, then we say $\psi_j$ is univariable instead of multivariable, where $|\bq|=\max_{1\le i\le n}|q_i|$. An approximating function $\psi:\Z^n\to\R_{\ge 0}$ is said to be non-increasing if
\[\psi(q_1,\dots,q_n)\ge\psi(q_1',\dots,q_n')\quad\text{whenever $|q_\ell|\le| q_\ell'|$ for $1\le \ell\le n$}.\]
For $\Psi=(\psi_1,\dots,\psi_m)$ with $\psi_j:\Z^n\to \R_{\ge 0}$ for $1\le j\le m$, define
\begin{equation}\label{eq:weighted}
	W(n,m;\Psi):=\big\{\bx\in[0,1]^{nm}:\|\bq\bx_j\|<\psi_j(\bq)\text{ $(1\le j\le m)$}\text{ for i.m. $\bq\in\Z^n$}\big\},
\end{equation}
and for $\psi:\Z^n\to\R_{\ge 0}$, define
\begin{equation}\label{eq:multiplicative}
	\cm^\times(n,m;\psi):=\bigg\{\bx\in[0,1]^{nm}:\prod_{j=1}^m\|\bq\bx_j\|<\psi(\bq)\text{ for i.m. $\bq\in\Z^n$}\bigg\},
\end{equation}
where $\|\cdot\|$ denotes the distance of a point in $\R$ to the nearest integer and i.m. stands for `{\em infinitely many}'.
Here, $\bx=(\bx_1,\dots,\bx_m)$ and $\bq$ are regarded as an $n\times m$ matrix and a row vector, respectively.  Throughout, the symbols $\ll$ and $\gg$ will be used to indicate an inequality with an unspecified positive multiplicative constant. By $a\asymp b$ we mean $a\ll b$ and $b\ll a$.

The modern study of the metric properties (in terms of Lebesgue measure and Hausdorff dimension/measure) of the sets $W(n,m;\Psi)$ was initiated by Khintchine \cite{Kh24} and Jarn\'ik \cite{Jarnik31}, who established elegant criteria, respectively, for determining the Lebesgue and Hausdorff measures of $W(1,1;\psi)$ under a reasonable monotonic assumption on $\psi$. The connection between the Fourier dimension and the metric Diophantine approximation emerged from the seminal work of Kaufman \cite{Ka81}, in which he provided the first explicit example of a set, namely $W(1,1;q\mapsto q^{-\tau})$ with $\tau>1$, whose Fourier and Hausdorff dimensions are the same. This paper mainly concerns the Lebesgue/Hausdorff measures and Fourier dimensions of $W(n,m;\Psi)$ and $\cm^\times(n,m;\psi)$. Although these two problems are closely related, they stem from distinct motivations. Therefore, we will introduce the relevant background and present the corresponding results separately in the following two subsections.

\subsection{Measure-theoretical dichotomy results}
Throughout,  a continuous, non-decreasing function $f$ defined on $\R^+$ and satisfying $\lim_{r\to 0^+}f(r)=0$ will be called a {\em dimension function}. For a dimension function $f$ and $s\ge 0$, by $f\preceq s$
we mean that
\begin{equation}\label{eq:prec}
	\frac{f(y)}{y^s}\le \frac{f(x)}{x^s}\quad\text{for any $0<x<y$},
\end{equation}
which implies that
\[\lim_{r\to 0^+}\frac{f(r)}{r^s}>0.\]
The limit could be finite or infinite. If the limit is infinite, then we write $f\prec s$ instead of $f\preceq s$. The notation $s\preceq f$ ($s\prec f$) is defined analogously.

For now, let us focus on the non-weighted case where $\psi_1=\cdots=\psi_m$, and we denote $\Psi$ by $\psi$. For monotonic univariable $\psi$, the fundamental Khintchine--Groshev Theorem  provides a criterion for determining whether the Lebesgue measure of $W(n,m;\psi)$ is zero or full, depending on whether a certain series converges or diverges. The necessity of the monotonicity assumption on $\psi$ has been one of the key questions in the metric Diophantine approximation. In the one-dimensional case ($n=m=1$), the
famous Duffin--Schaeffer counterexample shows that the monotonicity assumption on $\psi$
is absolutely necessary. For higher dimensional cases ($nm>1$), it was shown by the works of Schmidt \cite{Sc60} ($n\ge3$), Gallagher \cite{Ga65} ($n=1$ and $m>1$), and eventually Beresnevich and Velani \cite{BV10} ($n=2$) that the monotonicity assumption can be safely removed. Moreover, Beresnevich and Velani further showed the zero-one law remains valid even for multivariable approximating function.
\begin{thm}[{\cite{Kh24} and \cite[Theorem 5]{BV10}}]\label{t:BV}
	Let $n,m\ge1$ and let $\psi:\Z^n\to\R_{\ge 0}$ be an approximating function. Then,
		\[\lm^{nm}\big(W(n,m;\psi)\big)=\begin{dcases}
		0&\text{ if $\sum_{\bq\in\Z^n}\psi(\bq)^m<\infty$},\\
		1&\text{ if $\sum_{\bq\in\Z^n}\psi(\bq)^m=\infty$ and $\psi$ is non-increasing},
	\end{dcases}\]
	where $\lm^{nm}$ denotes the $nm$-dimensional Lebesgue measure. If $nm>1$, then the monotonicity assumption in the divergence part can be removed.
\end{thm}
The Lebesgue measure of $W(n,m;\psi)$ will be zero if the series in the theorem above converges. In this case, Hausdorff dimension and Hausdorff $f$-measure (denoted by  $\hdim$ and $\hm^f$, respectively) often provide a meaningful way to distinguish the relative `sizes' of these sets. Another fundamental result, due to Jarn\'ik \cite{Jarnik31}, provides a criterion analogous to Theorem \ref{t:BV} for determining the Hausdorff measure of $W(1,m;\psi)$. Although Hausdorff theory has been thought to be a subtle refinement of the Lebesgue theory, the remarkable `balls-to-balls' mass transference principle developed by Beresnevich and Velani \cite{BV06} allows us to transfer Lebesgue measure theoretic statements for lim sup sets defined by balls to Hausdorff statements, which means that Jarn\'ik's theorem is implied by Khintchine's theorem. Following this philosophy, Theorem \ref{t:BV}, together with the mass transference principle for systems of linear forms due to Allen and Beresnevich \cite{AB18} (a natural generalization of the original one), allows the Hausdorff measure of
$W(n,m;\psi)$ to be expressed in the following simple form.

\begin{thm}[{\cite{Jarnik31,AB18}}]\label{t:conAB}
	Let $n,m\ge 1$. Let $\psi:\Z^n\to\R_{\ge 0}$ be an approximating function. Let $f$ be a dimension function such that $(n-1)m\prec f\preceq nm$. Suppose that $\psi$ is non-increasing. Then,
	\[\hm^f\big(W(n,m;\psi)\big)=\begin{dcases}
		0&\text{ if $\sum_{\bq\in\Z^n\setminus\{0\}}f\bigg(\frac{\red{\psi(\bq)}}{|\bq|}\bigg)\bigg(\frac{\red{\psi(\bq)}}{|\bq|}\bigg)^{(1-n)m}|\bq|^m<\infty$},\\
		\hm^f([0,1]^{nm})&\text{ if $\sum_{\bq\in\Z^n\setminus\{0\}}f\bigg(\frac{\red{\psi(\bq)}}{|\bq|}\bigg)\bigg(\frac{\red{\psi(\bq)}}{|\bq|}\bigg)^{(1-n)m}|\bq|^m=\infty$}.
	\end{dcases}\]
	If $nm>1$, then the monotonicity assumption on $\psi$ can be removed.
\end{thm}

A natural but non-trivial generalization of the previously mentioned results is to consider the weighted theory, that is, $\psi_i\ne\psi_j$ for some $1\le i,j\le m$. The following result, partly due to Schmidt \cite{Sc60} and Hussain and Yusupova \cite{HY14}, provides a weighted analogue of Theorem \ref{t:BV}. For consistency, we present it in line with the setup of this paper.
\begin{thm}\label{t:Schmidt}
	Let $\Psi=(\psi_1,\dots,\psi_m)$ be an $m$-tuple of non-negative multivariable functions. Assume one of the following two conditions:
	\begin{enumerate}
		\item $n\ge 2$ and $\psi_j$ is univariable for $1\le j\le m$.
		\item there exists $\alpha\in\R$ such that for any $0<q_1<q_2\in\N$,
		\[q_1^\alpha \sum_{\bq_1\in\Z^n:|\bq_1|=q_1}\prod_{j=1}^m\psi_j(\bq_1)\gg q_2^\alpha \sum_{\bq_2\in\Z^n:|\bq_2|=q_2}\prod_{j=1}^m\psi_j(\bq_2).\]
	\end{enumerate}
	Then,
	\[\lm^{nm}\big(W(n,m;\Psi)\big)=\begin{dcases}
		0&\text{ if $\sum_{\bq\in\Z^n}\prod_{j=1}^m\psi_j(\bq)<\infty$},\\
		1&\text{ if $\sum_{\bq\in\Z^n}\prod_{j=1}^m\psi_j(\bq)=\infty$}.
	\end{dcases}\]

\end{thm}
\begin{rem}\label{r:mea}
	The univariable situation in item (1) was initially proved by Schmidt \cite{Sc60} for $n\ge 3$, and subsequently refined by Hussain and Yusupova \cite{HY14} for $n= 2$. For $n=1$, Gallagher \cite{Ga62} obtained the same result under the assumption similar to (2). For technical reasons  (see Remark \ref{r:reasonmon}), unlike the most commonly used monotonicity assumption, we require a nearly monotonic condition, as outlined in (2).
\end{rem}

In analogy to Theorem \ref{t:conAB}, if the series in the above theorem converges, it is natural to consider the size of the set under consideration in terms of Hausdorff dimension/measure.  The following statement was first proved by Rynne \cite{Ry98} for the case $n = 1$, and later generalized by Dickinson and Rynne \cite{RD00} to all $n \ge 1$. For analogues of these results in various settings---including the $p$-adics, complex numbers, quaternions, and formal power series---we refer the reader to \cite{GHSW23}.

\begin{thm}[{\cite{Ry98,RD00}}]
	Assume that for each $1\le j\le m$, there exists $0<\tau_j<\infty$ such that $\psi_j(\bq)=|\bq|^{-\tau_j}$, and that $\tau_1+\cdots+\tau_m>1$. Then,
	\[\hdim W(n,m;\Psi)=(n-1)m+\min_{1\le i\le m}\frac{m+n+\sum_{j:\tau_j<\tau_i}(\tau_i-\tau_j)}{1+\tau_i}.\]
\end{thm}
It is easy to prove that Dickinson and Rynne's theorem still holds if the order at infinity of each $\psi_j$ exists, i.e. $\lim_{q\to\infty}-\log\psi_j(q)/\log q<\infty$.
Regarding a general monotonic function $\Psi$, it turns out that Dickinson and Rynne's argument fails to apply, and the problem of determining the dimension of $W(n,m;\Psi)$ remained unsolved until recently. A breakthrough was achieved through the seminal work of Wang and Wu \cite{WW21}, in which they developed the so-called `rectangles-to-rectangles' mass transference principle. This result allows them to transfer a full Lebesgue measure property of a limsup set defined by a sequence of big rectangles
to the Hausdorff measure (the case $f(r)=r^s$) or Hausdorff dimension for the limsup set defined by shrinking the big rectangles to smaller ones.
As an application of this mass transference principle, they showed that $\hdim W(1,m;\Psi)$ depends on the set $\mathcal{U}(\Psi)$ of accumulation points of the sequence
\[
\bigg\{\bigg(-\frac{\log\psi_1(q)}{\log q},\dots, -\frac{\log\psi_m(q)}{\log q}\bigg)\bigg\},
\]
where $\infty$ is regarded as an accumulation point. However, applying their mass transference principle requires the assumption that the set $\mathcal{U}(\Psi)$ is bounded. A significant advancement was later made by Li, Liao, Velani, Wang, and Zorin \cite{LLVWZ24}, who extended Wang and Wu's mass transference principle to incorporate the `unbounded' case, thereby removing the necessity for $\mathcal{U}(\Psi)$ to be bounded.
\begin{thm}[{\cite{LLVWZ24,WW21}}]\label{t:dimension}
	Let $n=1$. For $1\le j\le m$, let $\psi_j:\Z\to\R^+$ be a positive and non-increasing approximating function. Then,
	\[\hdim W(1,m;\Psi)=\sup_{\bm\tau\in\mathcal U(\Psi)}\min\biggl\{\min_{i\in\mathcal L(\bm\tau)}\biggl\{\frac{m+1+\sum_{j:\tau_j<\tau_i}(\tau_i-\tau_j)}{1+\tau_i}\biggr\},\#\mathcal L(\bm\tau)\biggr\},\]
	where given $\bm\tau=(\tau_1,\dots,\tau_m)\in(\R^+\cup\{\infty\})^m$ we set $\red{\mathcal L(\bm\tau)}=\{1\le j\le m:\tau_j<\infty\}$.
\end{thm}
For a monotonic function $\Psi$, Theorem \ref{t:dimension} provides a satisfactory and complete description for the Hausdorff dimension of $W(1,m;\Psi)$. However, from the perspective of Theorem \ref{t:conAB}, two natural problems arise: the higher dimensional case ($n>1$) and the Hausdorff measure of $W(n,m;\Psi)$. As mentioned by the authors \cite{LLVWZ24}, their mass transference principle does not provide information about the Hausdorff measure, as it is established under the weaker full measure assumption. However, even with stronger assumptions---such as the original full measure property or even uniform local ubiquity, studied by the author \cite{He24} and Wang and Wu \cite{WW21}---it seems that it is still difficult to develop a general `rectangles-to-rectangles' mass transference principle that implies  Hausdorff measure statement, except for the case where $f(r)=r^s$.

The first result of this paper is a weighted analogue of Theorem \ref{t:conAB}. A crucial insight is that, in the weighted setup, our result is more appropriately understood as a consequence of the `balls-to-rectangles' mass transference principle, rather than the `rectangles-to-rectangles' mass transference principle. This distinction is crucial because the former principle yields desirable information about Hausdorff measure (see Theorem \ref{t:weaken}), but only under the original full measure assumption.
\begin{thm}\label{t:main}
	 Let $\Psi=(\psi_1,\dots,\psi_m)$ be an $m$-tuple of non-negative multivariable approximating functions. Let $f\prec nm$ be a dimension function such that  $(nm-a)\preceq f\preceq (nm-a+1)$ for some $\red{1\le a\le m}$. Assume one of the following two conditions:
	 \begin{enumerate}
	 	\item $n\ge 2$ and $\psi_j$ is univariable for $1\le j\le m$,
	 	\item $\psi_j$ is non-increasing for any $1\le j\le m$.
	 \end{enumerate}
	 Then,
		\[\hm^f\big(W(n,m;\Psi)\big)=\begin{dcases}
		0&\text{ if $\sum_{\bq\in\Z^n\setminus\{0\}}t_{\bq}(\Psi,f)|\bq|^m<\infty$},\\
		\hm^f([0,1]^{nm})&\text{ if $\sum_{\bq\in\Z^n\setminus\{0\}}t_{\bq}(\Psi,f)|\bq|^m=\infty$},
	\end{dcases}\]
	where
	\[t_{\bq}(\Psi,f)=	\min_{1\le i\le m}\bigg\{f\bigg(\frac{\psi_i(\bq)}{|\bq|}\bigg)\bigg(\frac{\psi_i(\bq)}{|\bq|}\bigg)^{(1-n)m}\prod_{j\in\ck_{\bq}(i)}\frac{\psi_j(\bq)}{\psi_i(\bq)}\bigg\},\]
	and in turn,
	\[\ck_{\bq}(i):=\{1\le j\le m:\psi_j(\bq)>\psi_i(\bq)\}.\]
\end{thm}
\begin{rem}
	The quantity $t_{\bq}(\Psi, f)$ represents the optimal cover, or equivalently the Hausdorff $f$-content, of neighborhoods of certain hyperplanes that naturally define the $\limsup$ set $W(n,m;\Psi)$. In particular, up to a multiplicative constant, it corresponds to the Hausdorff $f$-content of a hyperrectangle if $n=1$. Moreover, the condition $(m - a) \preceq f \preceq (m-a + 1)$ arises naturally, as it ensures that the Hausdorff $f$-content of a hyperrectangle has a `clearer' representation similar to Falconer's singular value function (see Proposition \ref{p:rec}).
\end{rem}
\begin{rem}
	Let us focus on the case $n=1$ for the moment. Although the set
	$W(1,m;\Psi)$ falls within the framework of the `rectangle-to-rectangle' mass transference principles developed in \cite{He24,WW21,LLVWZ24}, these principles fail to yield the Hausdorff measure statement. An important insight is that the Hausdorff measure of $W(1,m;\Psi)$ depends on the optimal cover of each individual hyperrectangle defining $W(1,m;\Psi)$, rather than on the optimal cover of the entire sequence of hyperrectangles. This partially suggests that the weighted case is more appropriately understood as a consequence of the `balls-to-rectangles' mass transference principle (see Lemma \ref{l:content>}).
\end{rem}
 Let us now turn to the multiplicative setup. Note that the set $\cm^\times(n,m;\psi)$ can almost be described as a union of $\limsup$ sets $W(n,m;\Psi)$ for some properly chosen functions $\Psi$. So, in many cases, the metric theory for $\limsup$ sets defined by hyperrectangles underpins that of multiplicative Diophantine approximation. The study of multiplicative Diophantine approximation is motivated by the famous Littlewood conjecture: for any pair $( x_1, x_2)\in[0,1]^2$,
 \begin{equation}\label{eq:little}
 	\liminf_{q\to\infty} q\|qx_1\|\|qx_2\|=0.
 \end{equation}
 Although Littlewood's conjecture remains largely unresolved, significant progress has been achieved in the study of its metric properties. For example, Einsiedler, Katok and Lindenstrauss \cite{EKL06} used their measure rigidity result to prove that the set of exceptions to Littlewood's conjecture is extremely small, namely is of zero Hausdorff dimension. Gallagher \cite{Ga62} proved that for $\lm^2$-almost every $(x_1,x_2)$, \eqref{eq:little} can be improved by a `$\log$ square' term. Another way to describe the sizes of $\cm^\times(1,2;\psi)$ is by considering its intersection with, for example, planar curves \cite{BV07matha}, and straight lines \cite{BHV20mem,Ch18duke,CY24Inv}. For more details we refer the reader to \cite{BaV11adv,BHV13acta,CT24mem,CT24adv,PV00acta,Yu23jam} and references therein.

For the Hausdorff measure, it was first shown by Beresnevich and Velani \cite{BV15recenttrends} that for monotonic $\psi$ and $s\in (m-1,m)$,
\[\hm^s\big(\cm^\times(1,m;\psi)\big)=\begin{cases}
	0&\text{if $\sum_{q=1}^\infty q^{m-s}\psi(q)^{s+1-m}\log^{m-2}q<\infty$},\\
	\hm^s([0,1]^m)&\text{if $\sum_{q=1}^\infty q^{m-s}\psi(q)^{s+1-m}=\infty$}.
\end{cases}\]
 There was a discrepancy in the above `$s$-volume' sum conditions for convergence and divergence when $m>2$, and they ask for a zero-full law to determine the Hausdorff measure for $\cm^\times(1,m;\psi)$. A satisfactory answer was given by Hussain and Simmons \cite{HS18JNT}, who proved that the $\log$ term in the above equation is redundant for any $m\ge 2$.

Our first result in the multiplicative setup naturally extends the existing results concerning the Lebesgue measure and the Hausdorff measure on the set $\cm^\times(1,m;\psi)$, which were obtained by Gallagher \cite{Ga62} and by Hussain and Simmons \cite{HS18JNT}, respectively, to systems of linear forms ($n > 1$). The Lebesgue result also answer a question raised by Hussain and Simmons \cite{HS18JNT} in the homogeneous case.

\begin{thm}\label{t:mulm}
	Let $n,m\ge1$ and let $\psi:\Z^n\to\R^+$ be an approximating function. Suppose that $\psi$ is non-increasing. Then,
	\[\lm^{nm}\big(\cm^\times(n,m;\psi)\big)=\begin{dcases}
		0&\text{ if $\sum_{\bq\in\Z^n}\psi(\bq)\log^{m-1}(\psi(\bq)^{-1})<\infty$},\\
		1&\text{ if $\sum_{\bq\in\Z^n}\psi(\bq)\log^{m-1}(\psi(\bq)^{-1})=\infty$}.
	\end{dcases}\]
\end{thm}
\begin{rem}
It was conjectured in \cite[Conjecture 4.1]{HS18JNT} that the Lebesgue measure of $\cm^\times(n,m;\psi)$ depends on the convergence or divergence of the series
	\[\sum_{\bq\in\Z^n}\psi(\bq)^{n}\log^{m-1}(\psi(\bq)^{-1})\footnote{In the presence of monotonicity, it is safe to replace $\log^{m-1}(\psi(\bq)^{-1})$ with $\log^{m-1}(|\bq|)$.
		}.\]
	However, our theorem shows that this is not the case.
\end{rem}
\begin{thm}[{\cite{HS18JNT,Zh19}}]\label{t:mulh}
	Let $n,m>1$. Let $\psi:\Z^n\to\R^+$ be an approximating function. Let $f$ be a dimension function such that $(n-1)m\prec f\preceq (nm-1+s)$ for some $0<s<1$. Suppose that $\psi$ is non-increasing. Then,
	\[\hm^f\big(\cm^\times(n,m;\psi)\big)=\begin{dcases}
		0&\text{ if $\sum_{\bq\in\Z^n\setminus\{0\}}f\bigg(\frac{\psi(\bq)}{|\bq|}\bigg)\bigg(\frac{\psi(\bq)}{|\bq|}\bigg)^{1-nm}|\bq|<\infty$},\\
		\hm^f([0,1]^{nm})&\text{ if $\sum_{\bq\in\Z^n\setminus\{0\}}f\bigg(\frac{\psi(\bq)}{|\bq|}\bigg)\bigg(\frac{\psi(\bq)}{|\bq|}\bigg)^{1-nm}|\bq|=\infty$}.
	\end{dcases}\]
\end{thm}
\begin{rem}\label{r:haum}
	 In \cite{HS18JNT}, the author proved the convergence part of the theorem and conjectured that the complementary divergence part should hold, which was later affirmatively answered by Zhang \cite{Zh19}. For the sake of completeness, we provide the proof here. We emphasize that the divergence part is, in fact, a consequence of the weighted setup (see Theorem \ref{t:main}). Let $\Psi=(\psi_1,\dots,\psi_m)$ be defined by $\psi_1=\cdots=\psi_{m-1}\equiv1$ and $\psi_m=\psi$. It will be shown in Section \ref{ss:divergencem} that $W(n,m;\Psi)\subset \cm^\times(n,m;\psi)$, and the set $W(n,m;\Psi)$ will be of full Hausdorff $f$-measure provided the series in the theorem diverges.
\end{rem}

The condition $f\preceq (nm-1+s)$ is crucial for eliminating an undesirable log term in the estimate of the upper bound of the Hausdorff $f$-measure (see Lemma \ref{l:mulplicativecover}). However, Theorem \ref{t:mulh} does not apply to the dimension function $f(r)=r^{nm}\log(1/r)$, it is therefore desirable to calculate the Hausdorff $f$-measure of $\cm^\times(n,m;\psi)$ for this case. It turns out that the $\log$ term would appear naturally, in line with the Lebesgue measure result described in Theorem \ref{t:mulm}.
\begin{thm}\label{t:mulhs}
	Let $n,m>1$. Let $\psi:\Z^n\to\R^+$ be an approximating function. Let $(nm-1)\preceq f\prec nm$ be a dimension function such that for any $0<t<1$ there exists $r_0=r_0(t)$ for which
	\begin{equation}\label{eq:fasy}
		\frac{f(\alpha r)}{f(r)}\asymp \alpha^{nm},\quad \text{for all $1<\alpha<r^{-t}$ and $\alpha r<r_0$},
	\end{equation}
	where the implied constant may depend on $t$.
	Suppose that $\psi$ is non-increasing. Then,
	\[\begin{split}
		&\hm^f\big(\cm^\times(n,m;\psi)\big)\\=&\begin{dcases}
			0&\text{ if $\sum_{\bq\in\Z^n\setminus\{0\}}f\bigg(\frac{\psi(\bq)}{|\bq|}\bigg)\bigg(\frac{\psi(\bq)}{|\bq|}\bigg)^{1-nm}|\bq|\log^{m-1}(\psi(\bq)^{-1})<\infty$},\\
			\hm^f([0,1]^{nm})&\text{ if $\sum_{\bq\in\Z^n\setminus\{0\}}f\bigg(\frac{\psi(\bq)}{|\bq|}\bigg)\bigg(\frac{\psi(\bq)}{|\bq|}\bigg)^{1-nm}|\bq|\log^{m-1}(\psi(\bq)^{-1})=\infty$}.
		\end{dcases}
	\end{split}\]
\end{thm}
\begin{rem}
	It is easily verified that $f(r)=r^{nm}\log(1/r)$ satisfies \eqref{eq:fasy}. Due to the presence of the extra $\log$ term, the strategy for obtaining the lower bound for $\hm^f(\cm^\times(n,m;\psi))$ described in Remark \ref{r:haum} is no longer applicable. To bypass it, similar to the proof in the weighted setup, the main idea is to bring the set $\cm^\times(n,m;\psi)$ into the framework of the `balls-to-rectangles' mass transference principle.
\end{rem}

\subsection{Fourier dimension and (non-)Salem sets}
The classical result due to Frostman provides a fundamental way to bound the Hausdorff dimension of a Borel set $E\subset\R^d$, which can also be expressed as the existence of finiteness of $s$-energy:
\[\hdim E=\sup\{s:\exists\mu\in\cp(E)\text{ such that }I_s(\mu)<\infty\},\]
where $\cp(E)$ is the set of all probability measures on $\R^d$ for which $\supp(\mu)$ is a compact subset of $E$, and $I_s(\mu)$ is the $s$-energy of $\mu$ defined by
\[I_s(\mu)=\int\int|\bx-\by|^{-s}\dif\mu(\bx)\dif\mu(\by).\]
Note that the inner integral can be written as the convolution
\[\int|\bx-\by|^{-s}\dif\mu(\bx)=h_s\ast\mu(\by)\quad\text{with $h_s(\bx)=|\bx|^{-s}$}.\]
This leads to a far from being trivial connection between the energy (and thus the Hausdorff dimension) and the Fourier transform:
\[I_s(\mu)\asymp \int_{\R^d}|\widehat{\mu}(\bx)|^2|\bx|^{s-d}\dif\bx,\]
where the implied constant depends on $d$ and $s$ only, and
\[\widehat{\mu}(\bx):=\int_{\R^d} e^{-2\pi \ic \bm\xi\cdot\bx}\dif\bm\xi.\]
Here, $\bm\xi\cdot\bx=\sum_{i=1}^d\xi_ix_i$ is the standard Euclidean inner product. This motivates the definition:
\begin{defn}[Fourier dimension and Salem set]
	For $E\subset \R^d$, the Fourier dimension of $E$ is defined by
	\[\fdim E:=\sup\{s\in[0,d]:\exists\mu\in\cp(E)\text{ such that }|\widehat{\mu}(\bx)|\ll |\bx|^{-s/2}\}.\]
	We say that $E$ is a Salem set if $\fdim E=\hdim E$.
\end{defn}
It follows from the above discussion that \[\fdim E\le \hdim E.\]
Hausdorff and Fourier dimensions are generally different. For example, every $k$-dimensional ($k<d$) hyperplane  has Hausdorff dimension $k$ and Fourier
dimension 0. The middle-third Cantor set has Hausdorff dimension $\log 2/ \log 3$ and Fourier dimension $0$.

Every ball in $\R^d$ is a Salem set of dimension $d$. Less trivially, every sphere in $\R^d$ is a Salem set of dimension $d-1$. The existence of Salem sets in $\R^d$ with dimension $s\ne0, d-1,d$ are more complicated. Historically, Salem sets are abundant in the random construction of sets. The first such random construction is due to Salem himself \cite{Sa51ark}; many subsequent
random constructions have appeared in \cite{Bl96ark,CS17can,Ek16,Ka66,LP09gafa,Mu18jfaa,SX06stu}. The first explicit construction of a Salem set results from a breakthrough by Kaufman \cite{Ka81}, who showed that $W(1,1;q\mapsto q^{-\tau})$ is a Salem set of dimension $2/(1+\tau)$ when $\tau>1$. The explicit examples of Salem sets in higher dimensions $\R^d$ were recently given by Hambrook \cite{Ha17} and Fraser and Hambrook \cite{FH23}, completely resolving a problem proposed by Kahane more than 50 years ago. It should be emphasized that the sets studied in \cite{Ha17,FH23} are quite different from the sets $W(n,m;\psi)$ studied in this paper. Instead, it was shown by Hambrook and Yu \cite{HY22}, and later improved by Cai and Hambrook \cite{CH24}, that the set $W(n,m;\psi)$ is far from being a Salem set unless $n=m=1$.
\begin{thm}[{\cite[Theorem 1.1]{HY22} and \cite[Theorem 1.4.1]{CH24}}]\label{t:fdim}
	Let $\psi:\Z^n\to\R_{\ge 0}$ be a multivariable approximating function (not necessarily monotonic). If $\sum_{\bq\in\Z^n}\psi(\bq)^m<\infty$, then
	\[\fdim W(n,m;\psi)=2s(\psi),\]
	where
	\[s(\psi)=\inf\bigg\{s\ge 0:\sum_{\bq\in\Z^n\setminus\{0\}}\bigg(\frac{\psi(\bq)}{|\bq|}\bigg)^s<\infty\bigg\}.\]
	In particular, $W(n,m;\psi)$ is not Salem unless $n=m=1$.
\end{thm}
	Our first theorem in this subsection concerns the following weighted generalization of the result above.
\begin{thm}\label{t:fdimW}
	Let $\Psi=(\psi_1,\dots,\psi_m)$ be an $m$-tuple of multivariable approximating functions with  $\psi_j:\Z^n\to\R_{\ge 0}$ for $1\le j\le m$ (not necessarily monotonic). Suppose that $\sum_{\bq\in\Z^n}\prod_{j=1}^{m}\psi_j(\bq)<\infty$ and that
	\begin{equation}\label{eq:tec}
		1>s(\Psi):=\inf\bigg\{s\ge 0:\sum_{\bq\in\Z^n\setminus\{0\}}\bigg(\min_{1\le j\le m}\frac{\psi_j(\bq)}{|\bq|}\bigg)^s<\infty\bigg\}.
	\end{equation}
	Then,
	\[\fdim W(n,m;\Psi)=2s(\Psi).\]
\end{thm}

Despite the technical condition \eqref{eq:tec} in the above statement, a phenomenon not presented by Theorem \ref{t:fdim} is as follows: Let $n=1$, $m=2$, $\psi_1(q) \equiv 1$ and $\psi_2(q) = q^{-2}$. Then, we have $W(1,2;\Psi) = [0,1] \times W(1,1;\psi_2)$, and
\[\frac{2}{3}=\fdim W(1,2;\Psi)=\fdim W(1,1;\psi_2),\]
which somewhat implies that the Fourier dimension of $W(1,2;\Psi)$ is `governed' by its `worst direction', namely the vertical direction. This is not accidental, as the Fourier dimension of product sets has a property quite different from that of Hausdorff dimension:
\[\fdim (E\times F)\ge\min\{\fdim E,\fdim F\},\]
and, as shown in \cite[Theorem 3.2]{Fr24}, the equality holds if both $E$ and $F$ are compact sets with zero Lebesgue measure (in the appropriate ambient dimension). Our next result shows that the equality indeed holds for all Borel sets.
\begin{thm}\label{t:product}
	Let $E\subset\R^{d-k}$ and $F\subset\R^k$ be two Borel sets with zero Lebesgue measure (in the appropriate ambient dimension). Then,
	\[\fdim(E\times F)=\min\{\fdim E,\fdim F\}.\]
\end{thm}
As a corollary, we obtain the Fourier dimension of the product of $\limsup$ sets.
\begin{cor}\label{c:prod}
	Let $\{W(n_\ell,m_\ell;\Psi_\ell)\}_{\ell=1}^k$ be a finite collection of $\limsup$ sets. Assume that for $1\le \ell\le k$, $W(n_\ell,m_\ell;\Psi_\ell)$ is of zero $(n_\ell m_\ell)$-dimensional Lebesgue measure. Then,
	\[\fdim\bigg(\prod_{\ell=1}^{k}W(n_\ell,m_\ell;\Psi_\ell)\bigg)=\min_{1\le \ell\le k}\fdim W(n_\ell,m_\ell;\Psi_\ell).\]
\end{cor}
\begin{rem}
	It is worth noting that the result due to Wang and Wu \cite{WW24tams} implies that
	\[\hdim\bigg(\prod_{\ell=1}^{k}W(n_\ell,m_\ell;\Psi_\ell)\bigg)=\min_{1\le \ell\le k}\bigg\{\hdim W(n_\ell,m_\ell;\Psi_\ell)+\sum_{j\ne\ell}n_jm_j\bigg\}.\]
	Since $W(n,m;\Psi)\subset \prod_{j=1}^{m}W(n,1;\psi_j)$, the corollary provides insight into why the Fourier dimension of the weighted set is governed by its `worst direction'.
\end{rem}
For the multiplicative setup, although $\cm^\times(n,m;\psi)$ can be almost described as a union of $\limsup$ sets $W(n,m;\Psi)$ for appropriately chosen functions $\Psi$, its Fourier dimension is not an immediate consequence of Theorem \ref{t:fdimW}, since the Fourier dimension is not even finitely stable unless each set involved is closed. See \cite[Example 8]{EPS15jfg}\footnote{In \cite{EPS15jfg}, the authors introduced several definitions of Fourier dimension, and the definition they referred to as the `compact Fourier dimension' is consistent with the one used in the current work.} for more details. Nevertheless, this philosophy remains applicable. The following result generalizes the work of Tan and Zhou \cite{TZ24}, in which the case $n=1$ and $m=2$ was established.
\begin{thm}\label{t:fdimM}
	Let $\psi:\Z^n\to\R_{\ge 0}$ be a multivariable approximating function (not necessarily monotonic). If $\sum_{\bq\in\Z^n}\psi(\bq)\log^{m-1}(\psi(\bq)^{-1})<\infty$, then
	\[\fdim\cm^\times(n,m;\psi)=2\tau(\psi)\]
	where
	\[\tau(\psi):=\inf\biggl\{\tau\ge 0:\sum_{\bq\in\Z^n\setminus\{0\}}^\infty\bigg(\frac{\psi(\bq)^{1/m}}{|\bq|}\bigg)^\tau<\infty\biggr\}.\]
\end{thm}
\begin{rem}\label{r:mulf}
	Let $W(n,m;\Psi)\subset \cm^\times(n,m;\psi)$ be defined in the same manner as in Remark \ref{r:haum}, where $\psi_1=\cdots=\psi_{m-1}=1$ and $\psi_m=\psi$. This set shares the same Hausdorff $f$-measure as $\cm^\times(n,m;\psi)$. However, this is not the case when we consider their Fourier dimensions. The set that has the same Fourier dimension as $\cm^\times(n,m;\psi)$ is  $W(n,m;\psi^{1/m})$, rather than $W(n,m;\Psi)$. This discrepancy can be explained by Theorem \ref{t:fdimW}, which shows that the Fourier dimension is determined by the `worst direction', preventing the Fourier dimension of $W(n,m;\Psi)$ from exceeding that of $W(n,m;\psi^{1/m})$. By Theorem \ref{t:fdim}, the lower bound for the Fourier dimension of $\cm^\times(n,m;\psi)$ follows immediately, however, as noted above, this does not provide an upper bound, since the Fourier dimension is not finitely stable.
\end{rem}
\begin{rem}
	Given that the bulk of the work in the above result lies in establishing an upper bound for $\fdim\cm^\times(n,m;\psi)$, it is worth noting that although the method developed by Tan and Zhou \cite{TZ24} is undoubtedly applicable to higher dimensions, the associated calculations grow increasingly complex as the dimension rises. Based on the observations outlined in Remark \ref{r:mulf}, we employ several simple yet effective tricks to bypass these calculations, thereby presenting a simpler proof.
\end{rem}

\subsection{Structure of the paper}
In Section \ref{s:Hausdorff}, we recall several notions and properties of Hausdorff measure and content. Additionally, we present the Hausdorff $f$-content of a hyperrectangle, which admits a representation similar to Falconer's singular value function. In the consecutive Sections \ref{s:Lebpart}--\ref{s:Fourierpart}, we prove the Lebesgue measure, the Hausdorff measure and the Fourier dimension of $W(n,m;\Psi)$, respectively. Analogous proofs for $\cm^\times(n,m;\Psi)$ are provided in Sections \ref{s:Lebpartm}--\ref{s:Fourierpartm}. In Section \ref{s:product}, we determine the Fourier dimension of product sets and present some applications.
	\section{Hausdorff measure and content}\label{s:Hausdorff}
Let $f$ be a dimension function. For a set $E\subset \R^d$ and $\eta>0$, let
\[\mathcal H_\eta^f(E)=\inf\bigg\{\sum_{i}f(|B_i|):E\subset \bigcup_{i\ge 1}B_i, \text{ where $B_i$ are balls with $|B_i|\le \eta$}\bigg\},\]
where $|A|$ denotes the diameter of $A$.
The {\em Hausdorff $f$-measure} of $E$ is defined as
\[\hm^f(E):=\lim_{\eta\to 0^+}\mathcal H_\eta^f(E).\]
When $\eta=\infty$, $\hc^f(E)$ is referred to as {\em Hausdorff $f$-content} of $E$.

The key tool for proving the divergence parts of Theorems \ref{t:main} and \ref{t:mulhs} is the following `balls-to-open sets' mass transference principle. Originally established by Koivusalo and Rams \cite{KR21} for Hausdorff dimension, this mass transference principle was later refined by Zhong \cite{Zh21} to extend to the Hausdorff measure statement.

\begin{thm}[{\cite[Theorem 3.2]{KR21} and \cite[Theorem 1.6]{Zh21}}]\label{t:weaken}
	Let $f$ be a dimension function such that $f\preceq d$. Assume that $\{B_k\}$ is a sequence of balls in $[0,1]^d$ with radii tending to 0, and that $\lm^d(\limsup B_k)=1$. Let $\{E_k\}$ be a sequence of open sets such that $E_k\subset B_k$. If there exists a constant $c>0$ such that for any $k\ge 1$,
	\begin{equation}\label{eq:cont}
		\hc^f (E_k)>c\lm^d(B_k),
	\end{equation}
	then,
	\[\hm^f\Big(\limsup_{k\to\infty} E_k\Big)=\hm^f([0,1]^d).\]
\end{thm}
\begin{rem}
	In fact, $\limsup$ set satisfying \eqref{eq:cont} has the so-called large intersection given in \cite[Theorem 2.4]{He24}, which means that the intersection of the sets with countably many similar copies of itself still has full Hausdorff $f$-measure.
\end{rem}
With this result now at our disposal, the main ideas to prove the divergence parts of Theorems \ref{t:main} and \ref{t:mulhs} are to verify certain sets under consideration satisfying some Hausdorff $f$-content bound. For this purpose, the following mass distribution principle will be crucial.
\begin{prop}[Mass distribution principle {\cite[Lemma 1.2.8]{BiPe17}}]\label{p:MDP}
	Let $ E $ be a Borel subset of $ \R^d $. If $ E $ supports a Borel probability measure $ \mu $ that satisfies
	\[\mu(B)\le cf(|B|),\]
	for some constant $ 0<c<\infty $ and for every ball $B$, then $ \hc^f(E)\ge1/c $.
\end{prop}
Given that the $\limsup$ sets under consideration are defined by hyperrectangles, it is necessary to estimate their Hausdorff $f$-content.
\begin{prop}\label{p:rec}
	Let $R\subset\R^d$ be a hyperrectangle with side lengths
	\[a_1\ge a_2\ge\cdots\ge a_d>0.\]
	Suppose that $f$ is a dimension function such that $k\preceq f\preceq (k+1)$ for some $0\le k\le d-1 $. Then,
	\[\hc^f(R)\asymp a_1\cdots a_k a_{k+1}^{-k} f(a_{k+1}),\]
	where the unspecified constant depends on $d$ only.
\end{prop}
\begin{proof}
	For the upper bound, in view of the definition of Hausdorff $f$-content, we consider covering $R$ by balls of radius $a_i$ ($1\le i\le d$), respectively. For any $1\le i\le d$, it is not difficult to see that $R$ can be covered by $\asymp\prod_{j=1}^{i-1}a_j/a_i $ balls of radius $a_i$. Thus,
	\[\hc^f(R)\ll \min_{1\le i\le d}\bigg\{f(a_i)\prod_{j=1}^{i-1}\dfrac{a_j}{a_i}\bigg\}=\min_{0\le i\le d-1}\{a_1\cdots a_ia_{i+1}^{-i}f(a_{i+1})\}.\]
	We next show that, up to a multiplicative constant, the reverse inequality holds, and leave the verification that the minimum is actually attained at $i=k$ at the end of the proof.

	To obtain the lower bound for the Hausdorff $f$-content, we first define a measure $\mu$ supported on $R$ as follows:
	\[\mu=\frac{\lm^d|_R}{\lm^d(R)}=\frac{\lm^d|_R}{a_1\cdots a_d}.\]
	Next, we estimate the $\mu$-measure of arbitrary balls $B(\bx,r)$ with $\bx\in R$ and $r>0$. We consider two cases.

	\noindent\textbf{Case 1:} $0<2r\le a_d$. Then,
	\[\mu\big(B(\bx,r)\big)\ll \frac{(2r)^d}{a_1\cdots a_d}.\]
	By the definition of the symbol $\preceq$, the condition  $k\preceq f\preceq (k+1)$ implies
	\begin{equation}\label{eq:<f<}
		i\preceq f \quad\text{for any $i\le k$},\qaq f\preceq j\quad\text{for any $j\ge k+1$}.
	\end{equation}
	Since $k+1\le d$, we have $f\preceq d$, and so
	\[\frac{f(a_d)}{a_d^d}\le\frac{f(2r)}{(2r)^d}\quad\Longrightarrow\quad (2r)^d\le \frac{a_d^df(2r)}{f(a_d)}.\]
	It follows that
	\[\mu\big(B(\bx,r)\big)\ll\frac{a_d^df(2r)}{f(a_d)a_1\cdots a_d}= \frac{f(2r)}{a_1\cdots a_{d-1}a_d^{-d+1}f(a_d)}.\]

	\noindent\textbf{Case 2:} $a_{i+1}< 2r\le a_i$ for some $1\le i\le d-1$. By definition,
		\[\mu\big(B(\bx,r)\big)\ll \frac{(2r)^ia_{i+1}\cdots a_d}{a_1\cdots a_d}=\frac{(2r)^i}{a_1\cdots a_i}.\]
	By \eqref{eq:<f<}, it holds that
	\[\text{either \quad$i\preceq f$\quad or\quad $f\preceq i$}.\]
	Suppose we are in the case $i\preceq f$. Since $a_{i+1}<2r$, we have
	\[\frac{(2r)^i}{f(2r)}\le\frac{a_{i+1}^i}{f(a_{i+1})}\quad\Longrightarrow\quad (2r)^i\le \frac{a_{i+1}^if(2r)}{f(a_{i+1})}.\]
	Therefore,
	\[\mu\big(B(\bx,r)\big)\ll\frac{a_{i+1}^if(2r)}{f(a_{i+1})a_1\cdots a_i}= \frac{f(2r)}{a_1\cdots a_ia_{i+1}^{-i}f(a_{i+1})}.\]

	Suppose for otherwise that $f\preceq i$. Using $2r\le a_i$, by similar calculations, we obtain
	\[\mu\big(B(\bx,r)\big)\ll\frac{a_i^if(2r)}{f(a_i)a_1\cdots a_i}= \frac{f(2r)}{a_1\cdots a_{i-1}a_i^{-i+1}f(a_i)}.\]

	Gathering the estimates in Cases 1 and 2, we have that for any $\bx\in R$ and $r>0$,
	\[\mu\big(B(\bx,r)\big)\ll \frac{f(2r)}{\min\limits_{0\le i\le d-1}a_1\cdots a_ia_{i+1}^{-i}f(a_{i+1})}.\]
	It follows from the mass distribution principle (see Proposition \ref{p:MDP}) that
\begin{equation}\label{eq:hcR>>}
	\hc^f(R)\gg \min_{0\le i\le d-1}\{a_1\cdots a_ia_{i+1}^{-i}f(a_{i+1})\},
\end{equation}
	as desired.

	It remains to show that the minimum is attained at $i=k$. For any $i\le k$, since $a_{i+1}\le a_i$, by $i\preceq f$ (see \eqref{eq:<f<}), we have $f(a_{i+1})a_{i+1}^{-i}\le f(a_i)a_i^{-i}$. Therefore,
	\[a_1\cdots a_{i-1}a_i^{-i+1}f(a_i)=a_1\cdots a_{i-1}a_i a_i^{-i}f(a_i)\ge a_1\cdots a_{i-1}a_i a_{i+1}^{-i}f(a_{i+1}),\]
	which means that
		\begin{equation}\label{eq:a1>a_k}
			f(a_1)\ge a_1 a_2^{-1}f(a_2)\ge\cdots\ge a_1\cdots a_ka_{k+1}^{-k}f(a_{k+1}).
		\end{equation}
		Analogously, for any $i\ge k+1$, since $a_{i+1}\le a_i$, by $f\preceq i$ (see \eqref{eq:<f<}), we have $f(a_i)a_i^{-i}\le f(a_{i+1})a_{i+1}^{-i}$. Therefore,
		\[a_1\cdots a_{i-1}a_i^{-i+1}f(a_i)=a_1\cdots a_{i-1}a_i a_i^{-i}f(a_i)\le a_1\cdots a_{i-1}a_i a_{i+1}^{-i}f(a_{i+1}),\]
		which means that
		\begin{equation}\label{eq:a_k<a_d}
			a_1\cdots a_ka_{k+1}^{-k}f(a_{k+1})\le \cdots\le a_1\cdots a_{d-1}a_d^{-d+1}f(a_d).
		\end{equation}
		Combining with \eqref{eq:a1>a_k} and \eqref{eq:a_k<a_d}, we conclude the proposition.
\end{proof}
\section{Lebesgue measure of $W(n,m;\Psi)$}\label{s:Lebpart}
Theorem \ref{t:Schmidt} (1) has already been established in \cite{Sc60,HY14}, so in this section, we focus on the remaining case of the theorem.
We make use of the following natural representation of $W(n,m;\Psi)$ as a $\limsup$ set. For $\bd=(\delta_1,\dots,\delta_m)\in(\R^+)^m$ and  $\bq=(q_1,\dots,q_n)\in\Z^n$, let
\begin{equation}\label{eq:rqdelta}
	R(\bq,\bd):=\big\{\bx\in[0,1]^{nm}:\|\bq\bx_j\|<\delta_j\text{ for $1\le j\le m$}\big\}.
\end{equation}
The convergence part of Theorem \ref{t:Schmidt} is relatively straightforward to establish and does not require any assumptions on $\Psi$.
\begin{lem}\label{l:meaweic}
	Let $n,m\ge 1$. Then,
	\[\sum_{\bq\in\Z^n}\prod_{j=1}^m\psi_j(\bq)<\infty\quad\Longrightarrow\quad \lm^{nm}\big(W(n,m;\Psi)\big)=0.\]
\end{lem}
\begin{proof}
	We can write $W(n,m;\Psi)$ as
	\[W(n,m;\Psi)=\limsup_{\bq\in\Z^n:|\bq|\to\infty} R\big(\bq,\Psi(\bq)\big)=\bigcap_{Q=1}^{\infty}\bigcup_{q=Q}^\infty\bigcup_{\bq\in\Z^n:|\bq|=q}R\big(\bq,\Psi(\bq)\big).\]
	Since each set $R(\bq,\Psi(\bq))$ has Lebesgue measure $\asymp\prod_{j=1}^{m}\psi_j(\bq)$, we have
	\[\begin{split}
		\lm^{nm}\big(W(n,m;\Psi)\big)&\ll\lim_{Q\to\infty}\sum_{q=Q}^{\infty}\lm^{nm}\bigg(\sum_{\bq\in\Z^n:|\bq|=q}R\big(\bq,\Psi(\bq)\big)\bigg)\\
		&\le\lim_{Q\to\infty}\sum_{q=Q}^{\infty}\sum_{\bq\in\Z^n:|\bq|=q}\prod_{j=1}^{m}\psi_j(\bq)=0,
	\end{split}\]
	where the last equality follows from convergence of the series in the lemma.
\end{proof}

The divergence part of the theorem is  more involved and relies on the following result, which does not carry any monotonicity
assumptions and allows us to reduce the proof to show that $W(n,m;\Psi)$ is of positive measure.
\begin{thm}[{\cite[Theorem 3.1]{Li13}}]\label{l:reduce}
	For any $n,m\ge 1$, let $\Psi=(\psi_1,\dots,\psi_m)$ be an $m$-tuple of multivariable approximating functions with $\psi_j:\Z^n\to \R^+$ for $1\le j\le m$. Then,
	\[\lm^{nm}\big(W(n,m;\Psi)\big)\in\{0,1\}.\]
	In particular,
	\[\lm^{nm}\big(W(n,m;\Psi)\big)>0\quad\Longrightarrow\quad \lm^{nm}\big(W(n,m;\Psi)\big)=1.\]
\end{thm}
Lamperti's result \cite{La63} provides a method for establishing lower bounds on the measure of $\limsup$ sets.
\begin{lem}[\cite{La63}]\label{l:quasiind}
	Let $(X,\mu)$ be a probability space, and let $\{E_i\}$ be measurable subsets of $X$ such that
	\[\sum_{i=1}^{\infty}\mu(E_i)=\infty\qaq\mu(E_i\cap E_j)\le C\mu(E_i)\mu(E_j)\quad\text{for some $C\ge 1$}.\]
	Then,
	\[\mu\Big(\limsup_{i\to\infty}E_i\Big)>0.\]
\end{lem}
In light of Lemma \ref{l:quasiind}, the desired statement follows by showing that the sets are quasi-independent and that the sum of their measures diverges. The weighted nature of $W(n,m;\Psi)$ suggests that it is more appropriate to study the quasi-independence of the following sets: For $\bd=(\delta_1,\dots,\delta_m)\in(\R^+)^m$ and  $\bq=(q_1,\dots,q_n)\in\Z^n$, let
\[\begin{split}
	R_j'(\bq,\delta_j):=&\big\{\bx_j\in[0,1]^{m}:|\bq\bx_j-p_j|<\delta_j\text{ for some $p_j\in\Z$ with $\gcd(\bq,p_j)$=1}\big\}
\end{split}\]
for $1\le j\le m$, and let
\begin{equation}\label{eq:R'}
	\begin{split}
		R'(\bq,\bd):=&\big\{\bx\in[0,1]^{nm}:|\bq\bx_j-p_j|<\delta_j\text{ ($1\le j\le m$)}\\
		&\hspace{8em}\text{for some $\bp\in\Z^m$ with $\gcd(\bq,p_j)=1$ for all $j$}\big\}.
	\end{split}
\end{equation}
Here $\gcd(\bq,p_j)$ denotes the greatest common divisor of $q_1,\dots,q_n,p_j$. Clearly, one has
\[R'\big(\bq,\Psi(\bq)\big)=\prod_{j=1}^mR_j'\big(\bq,\psi_j(\bq)\big).\]
We stress that in the definition of $R'(\bq,\bd)$, each $p_j$ is required to be coprime with $\bq$ instead of the weaker restriction $\gcd(\bq,\bp)=1$, which would result in a set smaller.
\begin{lem}[{\cite[Lemmas 8 and 9]{BV10}}]\label{l:BV}
	Let $n\ge 1$, $0\le\delta_1,\delta_2<1$ and $\bq_1\ne\pm\bq_2\in\Z^n$. For any $1\le j\le m$, we have
	\[\lm^n\big(R_j'(\bq_1,\delta_1)\big)=2\delta_1\frac{\varphi(d)}{d},\]
	where $d=\gcd(\bq_1)$ and $\varphi$ is the Euler quotient function. Moreover,
	\[\lm^n\big(R_j'(\bq_1,\delta_1)\cap R_j'(\bq_2,\delta_2)\big)\ll\delta_1\delta_2.\]
\end{lem}

Since $R'(\bq_1,\Psi(\bq_1))$ is the product of $R_j'(\bq_1,\psi_j(\bq_1))$, and observing that $\varphi(d)/d\ge \varphi(|\bq_1|)/|\bq_1|$, the next lemma follows directly.
\begin{lem}\label{l:mea}
	Let $n,m\ge 1$ and let $\bq_1,\bq_2\in\Z^n$ satisfying $\bq_1\ne\pm \bq_2$.
	Then,
	\[\bigg(\frac{\varphi(|\bq_1|)}{|\bq_1|}\bigg)^m\prod_{j=1}^{m}\psi_j(\bq_1)\le\lm^{nm}\big(R'\big(\bq_1,\Psi(\bq_1)\big)\big)\le \prod_{j=1}^{m}\psi_j(\bq_1)\]
	and
	\[\lm^{nm}\big(R'\big(\bq_1,\Psi(\bq_1)\big)\cap R'\big(\bq_2,\Psi(\bq_2)\big)\big)\ll\prod_{j=1}^{m}\psi_j(\bq_1)\psi_j(\bq_2).\]
\end{lem}
The next lemma shows that a sufficiently large portion of the sets $R'(\bq, \Psi(\bq))$ possess the expected measure.
\begin{lem}[{\cite[Lemma 3]{Ga62}}]\label{l:positivedensity}
	For each $k\ge 1$, the set $\Lambda(k)$ of positive integers $q$ for which
	\[\frac{\varphi(q)}{q}\ge 1-\frac{1}{k}\]
	has positive density.
\end{lem}

To be able to apply Lemma \ref{l:mea}, we exclude one of $\bq$ and $-\bq$ so that the remaining sum still comparable to the original one. For any $q\in\N$, let $\Gamma(q)$ be defined as
\begin{equation}\label{eq:gammaq}
	\Gamma(q):=\bigg\{\bq\in\Z^n:|\bq|=q\text{ and $\prod_{j=1}^{m}\psi_j(\bq)\ge \prod_{j=1}^{m}\psi_j(-\bq)$}\bigg\}.
\end{equation}
Obviously, one has
\[\sum_{\bq\in\Z^n:|\bq|=q}\prod_{j=1}^m\psi_j(\bq)\asymp \sum_{\bq\in\Gamma(q)}\prod_{j=1}^m\psi_j(\bq).\]
Recall that the condition described in Theorem \ref{t:Schmidt} (2) is that there exists $\alpha\in\R$ such that for any $0<q_1<q_2\in\N$,
\begin{equation}\label{eq:item2}
	q_1^\alpha \sum_{\bq_1\in\Z^n:|\bq_1|=q_1}\prod_{j=1}^m\psi_j(\bq_1)\gg q_2^\alpha \sum_{\bq_2\in\Z^n:|\bq_2|=q_2}\prod_{j=1}^m\psi_j(\bq_2).
\end{equation}

\begin{lem}\label{l:diverge}
	There exists a set $\Lambda\subset\N$ with positive density such that for any $\bq\in \Z^n$ with $|\bq|=q\in\Lambda$,
	\[2^{-m}\prod_{j=1}^{m}\psi_j(\bq)\le\lm^{nm}\big(R'\big(\bq,\Psi(\bq)\big)\big)\le \prod_{j=1}^{m}\psi_j(\bq).\]
	Furthermore, if Theorem \ref{t:Schmidt} (2) holds, then for such set $\Lambda$,
	\[\sum_{\bq\in\Z^n}\prod_{j=1}^m\psi_j(\bq)=\infty\quad\Longrightarrow\quad\sum_{q\in\Lambda}\sum_{\bq\in\Gamma(q)}\prod_{j=1}^m\psi_j(\bq)=\infty,\]
	where $\Gamma(q)$ is defined as in \eqref{eq:gammaq}.
\end{lem}
\begin{proof}
	Applying Lemma \ref{l:positivedensity} with $k=2$, we obtain a set $\Lambda=\Lambda(2)$ of positive integers $q$ such that
	\[\frac{\varphi(q)}{q}\ge \frac{1}{2},\]
	which together with Lemma \ref{l:BV} concludes the first point of the lemma.

	For the other property, observe that
	\[\sum_{\bq\in\Z^n}\prod_{j=1}^m\psi_j(\bq)=\sum_{k=0}^\infty\sum_{q=2^k}^{2^{k+1}-1}\sum_{\bq\in\Z^n:|\bq|=q}\prod_{j=1}^m\psi_j(\bq)=:\sum_{k=0}^\infty\sum_{q=2^k}^{2^{k+1}-1}S_q,\]
	where for $q\ge 1$,
	\[S_q:=\sum_{\bq\in\Z^n:|\bq|=q}\prod_{j=1}^m\psi_j(\bq).\]
	Now, suppose that Theorem \ref{t:Schmidt} (2) holds (see also \eqref{eq:item2}). For any $2^k\le q\le 2^{k+1}-1$,
	\begin{equation}\label{eq:>>}
		2^{k\alpha}S_{2^k}\gg q^\alpha S_q\gg 2^{\alpha(k+1)}S_{2^{k+1}}\quad\Longrightarrow\quad S_{2^k}\gg S_q\gg S_{2^{k+1}}.
	\end{equation}
	Therefore,
	\[\sum_{q=2^k}^{2^{k+1}-1}S_q\ll2^kS_{2^k}.\]
	On the other hand, since $\Lambda$ has positive density, for $k$ large enough we have
	\[\frac{\#\{q\in\Lambda:2^{k}\le q<2^{k+1}\}}{2^{k}}\ge\lambda>0,\]
	which together with \eqref{eq:>>} implies that
	\[\sum_{q\in\Lambda: 2^k\le q<2^{k+1}}S_q\gg \sum_{q\in\Lambda: 2^k\le q<2^{k+1}}S_{2^{k+1}}\gg 2^{k+1}S_{2^{k+1}}.\]
	It follows that
	\begin{equation*}
		\begin{split}
			\sum_{q\in\Lambda}\sum_{\bq\in\Gamma(q)}\prod_{j=1}^m\psi_j(\bq)&\asymp\sum_{q\in\Lambda}S_q=\sum_{k=0}^{\infty}\sum_{q\in\Lambda: 2^k\le q<2^{k+1}}S_q\gg \sum_{k=0}^{\infty}2^{k+1}S_{2^{k+1}}\\
			&=\sum_{k=1}^{\infty}2^{k}S_{2^{k}}\gg \sum_{k=1}^\infty\sum_{q=2^k}^{2^{k+1}-1}S_q=\infty.\qedhere
		\end{split}
	\end{equation*}
\end{proof}
The following lemma is a consequence of examining the quasi-independence of the sets $R(\bq,\Psi(\bq))$ with $\bq\in\Gamma(q)$ and $q\in\Lambda$.
\begin{lem}
	Assume Theorem \ref{t:Schmidt} (2). Then,
	\[\sum_{\bq\in\Z^n}\prod_{j=1}^m\psi_j(\bq)=\infty\quad\Longrightarrow\quad \lm^{nm}\big(W(m,m;\Psi)\big)=1.\]
\end{lem}
\begin{proof}
	In view of Lemma \ref{l:reduce}, it suffices to prove that the $\limsup$ set has positive measure. To this end, we will show that
	\[\lm^{nm}\bigg(\limsup_{q\in\Lambda:q\to\infty}\bigcup_{\bq\in\Gamma(q)}R'\big(\bq,\Psi(\bq)\big)\bigg)>0,\]
	where $\Lambda$ is taken from Lemma \ref{l:diverge} and $\Gamma(q)$ is defined as in \eqref{eq:gammaq}.
	By Lemma \ref{l:mea}, for any $\bq_1\in\Gamma(|\bq_1|)$ and $\bq_2\in\Gamma(|\bq_2|)$ with $|\bq_1|,|\bq_2|\in\Lambda$, we have
	\[\begin{split}
		&\lm^{nm}\big(R'\big(\bq_1,\Psi(\bq_1)\big)\cap R'\big(\bq_2,\Psi(\bq_2)\big)\big)
		\ll\prod_{j=1}^m\psi_j(\bq_1)\psi_j(\bq_2)\\
		\asymp& \lm^{nm}\big(R'\big(\bq_1,\Psi(\bq_1)\big)\big) \lm^{nm}\big(R'\big(\bq_2,\Psi(\bq_2)\big)\big),
	\end{split}\]
	which together with Lemma \ref{l:quasiind} concludes the lemma.
\end{proof}
\section{Hausdorff measure of $W(n,m;\Psi)$}\label{s:Hausdorffpart}
The proof of Theorem \ref{t:main} is naturally divided into two parts: convergence and divergence. We begin with the former, which is significantly easier to establish. Recall that
\[t_{\bq}(\Psi,f)=	\min_{1\le i\le m}\bigg\{f\bigg(\frac{\psi_i(\bq)}{|\bq|}\bigg)\bigg(\frac{\psi_i(\bq)}{|\bq|}\bigg)^{(1-n)m}\prod_{j\in\ck_{\bq}(i)}\frac{\psi_j(\bq)}{\psi_i(\bq)}\bigg\},\]
where
\[\ck_{\bq}(i):=\{1\le j\le m:\psi_j(\bq)>\psi_i(\bq)\}.\]
\subsection{Convergence part}\label{ss:convergence}
In this subsection, assume that
\begin{equation}\label{eq:concon}
	\sum_{\bq\in\Z^n\setminus\{0\}}t_{\bq}(\Psi,f)|\bq|^m<\infty.
\end{equation}
The convergence part of the theorem follows easily by considering its natural cover.
To obtain the Hausdorff measure of $W(n,m;\Psi)$, we rewrite it as
\[W(n,m;\Psi)=\bigcap_{Q=1}^\infty\bigcup_{q=Q}^\infty\bigcup_{\bq\in \Z^n:|\bq|=q}\bigcup_{\bp\in\Z^m:|\bp|\le q}\Delta\bigg(\rc_{\bq,\bp},\frac{\Psi(\bq)}{|\bq|}\bigg),\]
where for $(\bq,\bp)\in\Z^n\times\Z^m$,
\begin{equation}\label{eq:Deltaprod}
	\Delta\bigg(\rc_{\bq,\bp},\frac{\Psi(\bq)}{|\bq|}\bigg):=\prod_{j=1}^{m}\Delta\bigg(\rc_{\bq,p_j},\frac{\psi_j(\bq)}{|\bq|}\bigg),
\end{equation}
and in turn, for $1\le j\le m$,
\[\Delta\bigg(\rc_{\bq,p_j},\frac{\psi_j(\bq)}{|\bq|}\bigg):=\{\bx_j\in[0,1]^n:|\bq \bx_j-p_j|<\psi_j(\bq)\}.\]
It should be noted that for $1\le j\le m$, $\Delta(\rc_{\bq,p_j},\psi_j(\bq)/|\bq|)$ is simply the set of points in $[0,1]^n$ whose distance to the hyperplane $\bq \bx_j-p_j=0$ is less than $\psi_j(\bq)/|\bq|$. We will cover the set $\Delta(\rc_{\bq,\bp},\Psi(\bq)/|\bq|)$
by balls of radius $\psi_i(\bq)/|\bq|$, for $1\le i\le m$, respectively. Let $1\le i\le m$. Recall that
\[\ck_{\bq}(i):=\{1\le j\le m:\psi_j(\bq)>\psi_i(\bq)\}.\]
Along the $j$th direction ($1\le j\le m$), the $j$th component $\Delta(\rc_{\bq,p_j},\psi_j(\bq)/|\bq|)$ can be covered by
\[\asymp\begin{dcases}
	\bigg(\frac{\psi_i(\bq)}{|\bq|}\bigg)^{1-n}&\text{ if $j\notin\ck_{\bq}(i)$}\\
	\frac{\psi_j(\bq)}{|\bq|}\bigg(\frac{\psi_i(\bq)}{|\bq|}\bigg)^{-n}&\text{ if $j\in\ck_{\bq}(i)$}
\end{dcases}\]
balls of radius
\[\frac{\psi_i(\bq)}{|\bq|}.\]
By varying $i$, we observe that the Hausdorff $f$-content of $\Delta(\rc_{\bq,\bp},\Psi(\bq)/|\bq|)$
is at most
\[\begin{split}
	&\ll\min_{1\le i\le m}\bigg\{ f\bigg(\frac{\psi_i(\bq)}{|\bq|}\bigg)\prod_{j\notin\ck_{\bq}(i)}\bigg(\frac{\psi_i(\bq)}{|\bq|}\bigg)^{1-n}\prod_{j\in\ck_{\bq}(i)}\bigg(\frac{\psi_j(\bq)}{|\bq|}\bigg(\frac{\psi_i(\bq)}{|\bq|}\bigg)^{-n}\bigg)\bigg\}\\
	&= t_{\bq}(\Psi,f).
\end{split}\]
Note that for any $\bq\in\Z^n\setminus\{0\}$, there are $\asymp |\bq|^{m}$ choices of $\bp\in\Z^m$ for which $|\bp|\le |\bq|$.
By the definition of Hausdorff $f$-measure and the convergence of the series in \eqref{eq:concon},
\[\hm^f\big(W(n,m;\Psi)\big)\le \liminf_{q\to\infty}\sum_{\bq\in\Z^n:|\bq|\ge q}t_{\bq}(\Psi,f)|\bq|^m=0.\]
\subsection{Divergence part}
In this subsection, we will assume that $f$ is a dimension function such that
\[f\prec nm,\quad(nm-a)\preceq f\preceq(nm-a+1)\quad\text{for some $1\le a\le m$}\]
and
\[\sum_{\bq\in\Z^n\setminus\{0\}}t_{\bq}(\Psi,f)|\bq|^{m}=\infty.\]
Since $f\prec nm$, by the definition of the symbol $\prec$, we have for any $\bq\in\Z^n\setminus\{0\}$ with $|\bq|$ large enough,
\[f\bigg(\frac{\psi_i(\bq)}{|\bq|}\bigg)>\bigg(\frac{\psi_i(\bq)}{|\bq|}\bigg)^{nm},\quad\text{for }1\le i\le m.\]
Consequently,
\begin{equation}\label{eq:t>}
	t_{\bq}(\Psi,f)\ge\min_{1\le i\le m}\bigg\{\bigg(\frac{\psi_i(\bq)}{|\bq|}\bigg)^{m}\prod_{j\in\ck_{\bq}(i)}\frac{\psi_j(\bq)}{\psi_i(\bq)}\bigg\}>\prod_{j=1}^{m}\frac{\psi_j(\bq)}{|\bq|}.
\end{equation}
Since the $\limsup$ set $W(n,m;\Psi)$ does not depend on finitely many $\bq\in\Z^n\setminus\{0 \}$, here and hereafter we will always assume that \eqref{eq:t>} holds for all $\bq\in\Z^n\setminus\{0\}$.

For any $\bq\in\Z^n\setminus\{0\}$, rearrange $\psi_j(\bq)$ ($1\le j\le m$) in the following ascending order:
\begin{equation}\label{eq:ij}
	\psi_{i_1}(\bq)\le\psi_{i_2}(\bq)\le\cdots\le\psi_{i_m}(\bq).
\end{equation}
It should be noticed that $i_1,\dots,i_m$ depend on $\bq$, but the dependence is not specified here for the sake of simplicity. It is easily verified that
\[\prod_{j=1}^{m}\frac{\psi_{i_j}(\bq)}{|\bq|}\le\cdots\le\bigg(\frac{\psi_{i_k}(\bq)}{|\bq|}\bigg)^k\prod_{j=k+1}^{m}\frac{\psi_{i_j}(\bq)}{|\bq|}\le\cdots\le\bigg(\frac{\psi_{i_m}(\bq)}{|\bq|}\bigg)^m.\]
By \eqref{eq:t>}, there exists a unique integer $1\le k=k(\bq)\le m$ for which
\begin{equation}\label{eq:unique}
	\bigg(\frac{\psi_{i_k}(\bq)}{|\bq|}\bigg)^k\prod_{j=k+1}^{m}\frac{\psi_{i_j}(\bq)}{|\bq|}\le t_{\bq}(\Psi,f)<\bigg(\frac{\psi_{i_{k+1}}(\bq)}{|\bq|}\bigg)^{k+1}\prod_{j=k+2}^{m}\frac{\psi_{i_j}(\bq)}{|\bq|}.
\end{equation}
We ignore the right-hand side if $k=m$. Set
\[\varpi_{\bq}=\bigg(t_{\bq}(\Psi,f)\prod_{j=k+1}^m\frac{|\bq|}{\psi_{i_j}(\bq)}\bigg)^{1/k}.\]
Then, by \eqref{eq:unique},
\begin{equation}\label{eq:<o<}
	\dfrac{\psi_{i_k}(\bq)}{|\bq|}\le\varpi_{\bq}= \bigg(t_{\bq}(\Psi,f)\prod_{j=k+1}^m\frac{|\bq|}{\psi_{i_j}(\bq)}\bigg)^{1/k}<\dfrac{\psi_{i_{k+1}}(\bq)}{|\bq|}.
\end{equation}
Define an $m$-tuple $\Phi=(\phi_1,\dots,\phi_m)$ of multivariable approximating functions as follows: For any $\bq\in\Z^n\setminus\{0\}$,
\[\phi_{i_j}(\bq):=\begin{dcases}
	\psi_{i_j}(\bq)&\text{ if $j>k=k(\bq)$},\\
	|\bq|\cdot\varpi_{\bq}&\text{ if $j\le k=k(\bq)$}.
\end{dcases}\]
It follows from \eqref{eq:<o<} that
\begin{equation}\label{eq:mon}
	\phi_{i_1}(\bq)=\cdots=\phi_{i_k}(\bq)=|\bq|\cdot\varpi_{\bq}<\phi_{i_{k+1}}(\bq)\le\cdots\le\phi_{i_m}(\bq).
\end{equation}
\begin{lem}
	Let $\Phi$ be defined as above. Assume one of the following two conditions:
	\begin{enumerate}
		\item $n\ge 2$ and $\psi_j$ is univariable for $1\le j\le m$.
		\item $\psi_j$ is non-increasing for any $1\le j\le m$.
	\end{enumerate}
	Then,
	\[\lm^{nm}\big(W(n,m;\Phi)\big)=1.\]
\end{lem}
\begin{proof}
	Simple calculations show that
	\[\prod_{j=1}^m\phi_j(\bq)=\prod_{j=k+1}^m\psi_{i_j}(\bq)\prod_{j=1}^{k}|\bq|\cdot\varpi_{\bq}=t_{\bq}(\Psi,f)|\bq|^m.\]
	Since $\sum_{\bq\in\Z^n\setminus\{0\}}t_{\bq}(\Psi,f)|\bq|^m=\infty$, we have
	\[\sum_{\bq\in\Z^n\setminus\{0\}}\prod_{j=1}^{m}\phi_j(\bq)=\infty.\]
	If we are in situation (1), then the conclusion follows unconditionally from Theorem \ref{t:Schmidt} (1). Otherwise, based on Theorem \ref{t:Schmidt} (2), we will show that for any $0<q_1<q_2\in\N$,
	\[q_1^{-n-m+1}\sum_{\bq_1\in\Z^n:|\bq_1|=q_1}t_{\bq_1}(\Psi,f)q_1^m\gg q_2^{-n-m+1}\sum_{\bq_2\in\Z^n:|\bq_2|=q_2}t_{\bq_2}(\Psi,f)q_2^m.\]
	It is not difficult to see that for any $\bq_1$ with $|\bq_1|=q_1$, there exists a set $A(\bq_1)\subset\{\bq_2\in\Z^n:|\bq_2|=q_2\}$ such that
	\[\# A(\bq_1)\asymp \bigg(\frac{q_2}{q_1}\bigg)^{n-1},\quad |\bq_1^{(\ell)}|\le|\bq_2^{(\ell)}|\quad \text{for $1\le \ell\le n$ and $\bq_2\in A(\bq_1)$}\]
	and
	\[\{\bq_2\in\Z^n:|\bq_2|=q_2\}=\bigcup_{|\bq_1|=q_1}A(\bq_1),\]
	where $\bq_1^{(\ell)}$ is the $\ell$th coordiante of $\bq_1$. For example, for any $\bq_1\in \Z^n$ with $|\bq_1^{(\ell)}|=|\bq_1|$, we can set
	\[A(\bq_1)=\bigg\{\bq_2\in\Z^n:|\bq_2^{(\ell)}|=q_2\text{ and $|\bq_2^{(j)}|\in \bigg[\frac{|\bq_1^{(j)}|q_2}{q_1},\frac{(|\bq_1^{(j)}|+1)q_2}{q_1}\bigg]$ for $j\ne\ell$}\bigg\}.\]
	Notice that $t_{\bq}(\Psi,f)$ is the quantity that denotes $f$-volume of the set of $\Delta(\rc_{\bq,\bp},\Psi(\bq)/|\bq|)$, obtained by covering it with balls of radius $\psi_i(\bq)/|\bq|$ for some $1\le i\le m$ (see the discussion presented in Section \ref{ss:convergence}). By the monotonicity of $\Psi$, we have
	\[t_{\bq_1}(\Psi,f)\gg t_{\bq_2}(\Psi,f)\quad\text{for $\bq_1\in\Z^n$ with $|\bq_1|=q_1$ and $\bq_2\in A(\bq_1)$},\]
	and so
	\[t_{\bq_1}(\Psi,f)\gg \bigg(\frac{q_1}{q_2}\bigg)^{n-1}\sum_{\bq_2\in A(\bq_1)}t_{\bq_2}(\Psi,f).\]
	Summing over $\bq_1\in\Z^n$ with $|\bq_1|=q_1$, we have
	\[q_1^{-n+1}\sum_{\bq_1\in\Z^n:|\bq_1|=q_1}t_{\bq_1}(\Psi,f)\gg q_2^{-n+1} \sum_{\bq_2\in\Z^n:|\bq_2|=q_2}t_{\bq_2}(\Psi,f),\]
	or equivalently
	\[q_1^{-n-m+1}\sum_{\bq_1\in\Z^n:|\bq_1|=q_1}t_{\bq_1}(\Psi,f)q_1^m\gg q_2^{-n-m+1} \sum_{\bq_2\in\Z^n:|\bq_2|=q_2}t_{\bq_2}(\Psi,f)q_2^m.\]
	Therefore, by Theorem \ref{t:Schmidt} (2), the $\limsup$ set $W(n,m;\Phi)$ has full Lebesgue measure.
\end{proof}
\begin{rem}\label{r:reasonmon}
	The function $\Phi$  is not guaranteed to be monotonic, even though $\Psi$ is, which is why we require a nearly monotonic assumption in Theorem \ref{t:Schmidt} (2).
\end{rem}
Write
\[\begin{split}
	W(n,m;\Phi)&=\bigcap_{Q=1}^\infty\bigcup_{q=Q}^\infty\bigcup_{\bq\in\Z^n:|\bq|=q}\bigcup_{\bp\in\Z^m:|\bp|\le q}\Delta\bigg(\rc_{\bq,\bp},\frac{\Phi(\bq)}{|\bq|}\bigg)\\
	&=\bigcap_{Q=1}^\infty\bigcup_{q=Q}^\infty\bigcup_{\bq\in\Z^n:|\bq|=q}\bigcup_{\bp\in\Z^m:|\bp|\le q}\prod_{j=1}^{k}\Delta(\rc_{\bq,p_{i_j}},\varpi_{\bq})\prod_{j=k+1}^{m}\Delta\bigg(\rc_{\bq,p_{i_j}},\frac{\psi_{i_j}(\bq)}{|\bq|}\bigg),
\end{split}\]
where $i_j$ and $k$ are defined as in \eqref{eq:ij} and \eqref{eq:unique}, respectively.
For each $(\bq,\bp)\in\Z^n\times\Z^m$, by \eqref{eq:mon} the product set
\[\prod_{j=1}^{k}\Delta(\rc_{\bq,p_{i_j}},\varpi_{\bq})\prod_{j=k+1}^{m}\Delta\bigg(\rc_{\bq,p_{i_j}},\frac{\psi_{i_j}(\bq)}{|\bq|}\bigg)\]
can be covered by a finite collection $\mathcal C(\bq,\bp)$ of balls with radius $\varpi_{\bq}$ and centered at this set. It holds that
\[W(n,m;\Phi)\subset\bigcap_{Q=1}^\infty\bigcup_{q=Q}^\infty\bigcup_{\bq\in\Z^n:|\bq|=q}\bigcup_{\bp\in\Z^m:|\bp|\le q}\bigcup_{B\in\mathcal{C}(\bq,\bp)}B.\]
Therefore, the $\limsup$ set defined by balls in the right still has full Lebesgue measure. On the other hand,
\[W(n,m;\Psi)\supset\bigcap_{Q=1}^\infty\bigcup_{q=Q}^\infty\bigcup_{\bq\in\Z^n:|\bq|=q}\bigcup_{\bp\in\Z^m:|\bp|\le q}\bigcup_{B\in\mathcal{C}(\bq,\bp)}\bigg(B\cap\prod_{j=1}^{m}\Delta\bigg(\rc_{\bq,p_{i_j}},\frac{\psi_{i_j}(\bq)}{|\bq|}\bigg)\bigg).\]
In view of Theorem \ref{t:weaken}, to conclude $W(n,m;\Psi)$ has full Hausdorff $f$-measure, it suffices to prove that for any $\mathcal C(\bq,\bp)$ and any ball $B\in\mathcal C(\bq,\bp)$,
\begin{equation}\label{eq:cont>}
	\hc^f\bigg(B\cap\prod_{j=1}^{m}\Delta\bigg(\rc_{\bq,p_{i_j}},\frac{\psi_{i_j}(\bq)}{|\bq|}\bigg)\bigg)\gg\lm^{nm}(B)\asymp\varpi_{\bq}^{nm}.
\end{equation}

Since the radius of $B$ is $\varpi_{\bq}$, it is not difficult to deduce from \eqref{eq:ij} and \eqref{eq:<o<} that
\[B\cap\prod_{j=1}^{m}\Delta\bigg(\rc_{\bq,p_{i_j}},\frac{\psi_{i_j}(\bq)}{|\bq|}\bigg)\]
contains a hyperrectangle, denoted by $R=R(B,\bq,\bp)$, whose side lengths are
\[\frac{\psi_{i_1}(\bq)}{|\bq|},\underbrace{\varpi_{\bq},\dots,\varpi_{\bq}}_{n-1},\frac{\psi_{i_2}(\bq)}{|\bq|},\underbrace{\varpi_{\bq},\dots,\varpi_{\bq}}_{n-1},\dots,\frac{\psi_{i_k}(\bq)}{|\bq|},\underbrace{\varpi_{\bq},\dots,\varpi_{\bq}}_{n-1},\underbrace{\varpi_{\bq},\dots,\varpi_{\bq}}_{n(m-k)}.\]
The Hausdorff $f$-content of $R$ can be estimated as follows, which together with \eqref{eq:cont>} completes the divergence part of Theorem \ref{t:main}.
\begin{lem}\label{l:content>}
	Let $R=R(B,\bq,\bp)$ be given as above. We have
	\[\hc^f(R)\gg \varpi_{\bq}^{nm},\]
	where the implied constant is absolute. Consequently,
	\[\hm^f\big(W(n,m;\Psi)\big)=\hm^f([0,1]^{nm}).\]
\end{lem}
\begin{proof}
	Applying \eqref{eq:hcR>>} in Proposition \ref{p:rec} yields
	\[\hc^f(R)\asymp \min\bigg\{\min_{1\le \ell\le k}\biggl\{f\bigg(\frac{\psi_{i_\ell}(\bq)}{|\bq|}\bigg)\bigg(\frac{|\bq|\cdot \varpi_{\bq}}{\psi_{i_\ell}(\bq)}\bigg)^{nm-k}\prod_{j=\ell+1}^{k}\frac{\psi_{i_j}(\bq)}{\psi_{i_\ell}(\bq)}\biggr\},f(\varpi_{\bq})\bigg\}.\]
	If the minimum is attained at $f(\varpi_{\bq})$, then by $f\prec nm$,
	\[f(\varpi_{\bq})\gg \varpi_{\bq}^{nm},\]
	which is what we seek.

	Now, suppose that the minimum is attained for some $1\le \ell\le k$.
	Direct calculations show that
	\begin{align}
		&f\bigg(\frac{\psi_{i_\ell}(\bq)}{|\bq|}\bigg)\bigg(\frac{|\bq|\cdot \varpi_{\bq}}{\psi_{i_\ell}(\bq)}\bigg)^{nm-k}\prod_{j=\ell+1}^{k}\frac{\psi_{i_j}(\bq)}{\psi_{i_\ell}(\bq)}\notag\\
		=&\bigg(\varpi_{\bq}^{nm-k}\prod_{j=k+1}^{m}\frac{|\bq|}{\psi_{i_j}(\bq)}\bigg) f\bigg(\frac{\psi_{i_\ell}(\bq)}{|\bq|}\bigg)\bigg(\frac{\psi_{i_\ell}(\bq)}{ |\bq|}\bigg)^{(1-n)m}\prod_{j=\ell+1}^{m}\frac{\psi_{i_j}(\bq)}{\psi_{i_\ell}(\bq)}.\label{eq:ww}
	\end{align}
	Since $\psi_{i_1}(\bq)\le\cdots\le\psi_{i_m}(\bq)$,
	\[t_{\bq}(\Psi,f)=\min_{1\le \ell\le m}\bigg\{f\bigg(\frac{\psi_{i_\ell}(\bq)}{|\bq|}\bigg)\bigg(\frac{\psi_{i_\ell}(\bq)}{|\bq|}\bigg)^{(1-n)m}\prod_{j=\ell+1}^m\frac{\psi_{i_j}(\bq)}{\psi_{i_\ell}(\bq)}\bigg\}.\]
	Then, \eqref{eq:ww} is at least
	\[\ge \varpi_{\bq}^{nm-k}\bigg(\prod_{j=k+1}^{m}\frac{|\bq|}{\psi_{i_j}(\bq)}\bigg) t_{\bq}(\Psi,f)=\varpi_{\bq}^{nm},\]
	where we use the definition of $\varpi_{\bq}$ in the last equality. The last conclusion in the lemma follows from \eqref{eq:cont>} and Theorem \ref{t:weaken}.
\end{proof}
\section{Fourier dimension of $W(n,m;\Psi)$}\label{s:Fourierpart}
Recall that
\[s(\Psi):=\inf\bigg\{s\ge 0:\sum_{\bq\in\Z^n\setminus\{0\}}\bigg(\min_{1\le j\le m}\frac{\psi_j(\bq)}{|\bq|}\bigg)^s<\infty\bigg\}.\]
In this section, under the assumptions that $s(\Psi)<1$ and $\sum_{\bq\in\Z^n}\prod_{j=1}^{m}\psi_j(\bq)<\infty$, we prove that
\[\fdim W(n,m;\Psi)=2s(\Psi).\]

\subsection{Upper bound for $\fdim W(n,m;\Psi)$}
To avoid an abundance of notation, here we also use $\bx$ or $\bx_j$ to denote an $n$-dimensional vector. When there is a risk of ambiguity, the specific dimension will be clarified in the context.

For $\bd=(\delta_1,\dots,\delta_m)\in (\R_{\ge 0})^{m}$ and  $\bq=(q_1,\dots,q_n)\in\Z^n$, define
\[R_j(\bq,\delta_j):=\{\bx_j\in[0,1]^n:\|\bq\bx_j\|\le\delta_j\}\]
for $1\le j\le m$ and
\[R(\bq,\bd):=\prod_{j=1}^{m}R_j(\bq,\delta_j).\]
For $1\le j\le m$ and $\delta_j=0$, let $\lm_{\bq,j}$ denote the surface measure on $R_j(\bq,0)$, which is a union of hyperplanes.
\begin{lem}[{\cite[Equation (4)]{HY22}}]\label{l:fouriersurface}
	Let $1\le j\le m$. For each $\bk\in\Z^n$, we have
	\[|\widehat{\lm_{\bq,j}}(\bk)|=\begin{cases}
		|\bq|_2&\text{if $\bk=t\bq$ for some $t\in\Z$},\\
		0&\text{otherwise},
	\end{cases}\]
	where $|\bq|_2=\sqrt{q_1^2+\cdots+q_n^2}$.
\end{lem}
The upper bound will utilize the following lemma, which mirrors \cite[Theorem 2.3]{HY22}.
\begin{lem}\label{l:meaupp}
	Let $n,m\ge 1$. Let $\bq\in\Z^n\setminus\{0\}$ and $\bd\in(0,1/2)^m$. Let $\mu$ be a Borel probability measure supported on $[0,1]^{nm}$. Then,
	\[\mu\big(R(\bq,\bd)\big)\ll\delta_1\cdots\delta_m\Bigg(1+\sum_{\substack{\bt\in\Z^m\setminus\{0\}\\ \forall1\le j\le m, |t_j|\le 2/\delta_j}}|\widehat{\mu}(t_1\bq,\dots,t_m\bq)|\Bigg),\]
	where the implied constant is absolute.
\end{lem}
\begin{proof}
	Note that
	\[\mu\big(R(\bq,\bd)\big)=\int_{[0,1]^{nm}}\prod_{j=1}^{m}\cha_{R_j(\bq,\delta_j)}(\bx_j)\dif\mu(\bx),\]
	where $\bx=(\bx_1,\dots,\bx_m)\in[0,1]^{nm}$. To bound the measure expressed above, for each $1\le j\le m$, we will approximate the indicate funcion $\cha_{R_j(\bq,\delta_j)}$ by smooth function.

	Let $\phi$ be a non-negative Schwartz function on $\R^n$ such that
	\begin{equation}
		\widehat{\phi}(0)=1,\quad \widehat{\phi}(\bx)=0\quad\text{for $\bx\notin B(0,2)\subset\R^n$},
	\end{equation}
	and
	\begin{equation}
		m(\phi):=\min\{\phi(\bx):\bx\in B(0,2)\subset \R^n\}>0.
	\end{equation}
	For $\bx\in\R^n$ and $\gamma>0$, define $\phi_{\gamma}(\bx)=\phi(\bx/\gamma)$. Note that $\widehat{\phi_{\gamma}}(\bx)=\gamma^n\widehat{\phi}(\gamma\bx)$ for all $\bx\in\R^n$, and that $\phi_{\gamma}\ast\lm_{\bq,j}$ is a smooth $\Z^n$-periodic function on $\R^n$.

	Write $\bm\gamma=(\gamma_1,\dots,\gamma_m)$ with $\gamma_j=\delta_j/|\bq|_2$. Before proving the lemma, we begin with the fact that for $1\le j\le m$, $\cha_{R_j(\bq,\delta_j)}$ can be approximated by $\phi_{\gamma_j}\ast\lm_{\bq,j}$. For any $\bx\in R_j(\bq,\delta_j)=R_j(\bq,0)+B(0,\gamma_j)$, there exists $\bz\in R_j(\bq,0)$ such that $|\bx-\bz|\le\gamma_j$. Hence, for every $\by\in\R^n$
	\[\phi_{\gamma_j}(\bx-\by)\ge m(\phi)\cha_{B(\bx,2\gamma_j)}(\by)\ge m(\phi)\cha_{B(\bz,\gamma_j)}(\by).\]
	By the definition of convolution, it follows that for all $\bx\in\R^n$,
	\begin{equation}\label{eq:smoapp}
		\phi_{\gamma_j}\ast\lm_{\bq,j}(\bx)=\int_{[0,1]^n}\phi_{\gamma_j}(\bx-\by)\dif\lm_{\bq,j}(\by)\gg m(\phi)\gamma_j^{n-1}\cha_{R_j(\bq,\delta_j)}(\bx),
	\end{equation}
	where in the last step we use $\lm_{\bq,j}(B(\bz,\gamma_j))\asymp\gamma_j^{n-1}$ for all $\bz\in R_j(\bq,0)$.

	By Parseval's theorem (see e.g. \cite[(3.64)]{Ma15}), it follows from Lemma \ref{l:fouriersurface} that
	\begin{align}
		&\int_{[0,1]^{nm}}\prod_{j=1}^{m}\phi_{\gamma_j}\ast\lm_{\bq,j}(\bx_j)\dif\mu(\bx)\notag\\
		=&\sum_{(\bk_1,\dots,\bk_m)\in\Z^{nm}}\overline{\widehat{\mu}(\bk_1,\dots,\bk_m)}\cdot\prod_{j=1}^m\widehat{\phi_{\gamma_j}}(\bk_j)\widehat{\lm_{\bq,j}}(\bk_j)\notag\\
		\le&\gamma_1^n\cdots\gamma_m^n\sum_{\bt\in\Z^m}|\widehat{\mu}(t_1\bq,\dots,t_m\bq)|\cdot\prod_{j=1}^m\big(|\widehat{\phi}(\gamma_jt_j\bq)|\cdot|\bq|_2\big)\notag\\
		\le&\gamma_1^n\cdots\gamma_m^n|\bq|_2^m\bigg(1+\sum_{\substack{\bt\in\Z^m\setminus\{0\}\\ \forall1\le j\le m, |t_j|\le 2/\delta_j}}|\widehat{\mu}(t_1\bq,\dots,t_m\bq)|\Bigg),\notag
	\end{align}
	where we use the facts that $|\widehat{\phi}|\le\widehat{\phi}(0)=1$ and $\widehat{\phi}=0$ outside $B(0,2)$ in the last inequality.
		Combining this with \eqref{eq:smoapp} yields
		\[\begin{split}
			\mu\big(R(\bq,\bd)\big)\ll \delta_1\cdots\delta_m\Bigg(1+\sum_{\substack{\bt\in\Z^m\setminus\{0\}\\ \forall1\le j\le m, |t_j|\le 2/\delta_j}}|\widehat{\mu}(t_1\bq,\dots,t_m\bq)|\Bigg).\qedhere
		\end{split}\]
	\end{proof}
	Now, we prove the upper bound for $\fdim W(n,m;\Psi)$.
	\begin{proof}[Proof of Theorem \ref{t:fdimW}: upper bound]
		Assume for the contrary that there exists a measure $\mu$ supported on $W(n,m;\Psi)$ and $s>s(\Psi)$ such that $|\widehat{\mu}(\bx)|\ll|\bx|^{-s}$ for all $\bx\in \R^{nm}$. Since $s(\Psi)<1$ (by our assumption), we may assume that $s<1$ as well. By Lemma \ref{l:meaupp}, for any $\bd\in(0,1/2)^m$,
		\[\mu\big(R(\bq,\bd)\big)\ll \delta_1\cdots\delta_m\Bigg(1+\sum_{\substack{\bt\in\Z^m\setminus\{0\}\\ \forall1\le j\le m, |t_j|\le 2/\delta_j}}|\widehat{\mu}(t_1\bq,\dots,t_m\bq)|\Bigg)\]
		The Fourier decay assumption implies that the inner-most sum over $\bt$ is at most
		\[\ll |\bq|^{-s}\sum_{\substack{\bt\in\Z^m\setminus\{0\}\\ \forall1\le j\le m, |t_j|\le 2/\delta_j}}|\bt|^{-s}=:|\bq|^{-s}S_{\bq}.\]
		Next, we estimate the summation $S_\bq$. Assume without loss of generality that $\delta_1\ge \cdots\ge \delta_m$. Since $\bt\ne 0$, with the convention $1/0=\infty$, we have \[|\bt|^{-s}=\min\{1,|t_1|^{-s},\dots,|t_m|^{-s}\}\le \min\{1,|t_m|^{-s}\}.\]
		Therefore,
		\[\begin{split}
			S_\bq\le\sum_{\substack{\bt\in\Z^m\setminus\{0\}\\ \forall1\le j\le m, |t_j|\le 2/\delta_j}}\min\{1,|t_m|^{-s}\}&\asymp\sum_{0\le |t_m|\le 2/\delta_m}(\delta_1\cdots\delta_{m-1})^{-1}\cdot\min\{1,|t_m|^{-s}\}\\
			&\asymp (\delta_1\cdots\delta_{m-1})^{-1}\delta_m^{s-1},
		\end{split}\]
		where we use $s<1$ in the last step.
		Substitute this upper bound, one has
		\[\begin{split}
			\mu\big(R(\bq,\bd)\big)&\ll \delta_1\cdots\delta_m\big(1+O(|\bq|^{-s}(\delta_1\cdots\delta_{m-1})^{-1}\delta_m^{s-1})\big)\\
			&=\delta_1\cdots\delta_m+O(|\bq|^{-s}\delta_m^s).
		\end{split}\]
		Setting $\bd=\Psi(\bq)=(\psi_1(\bq),\dots,\psi_m(\bq))$ and summing over $\bq\in\Z^n\setminus\{0\}$, and note that $\delta_m=\min_{1\le j\le m}\psi_j(\bq)$ according to the above discussion, we get
		\[\sum_{\bq\in\Z^n\setminus\{0\}}\mu\big(R(\bq,\Psi(\bq))\big)\ll\sum_{\bq\in\Z^n\setminus\{0\}}\Bigg(\prod_{j=1}^{m}\psi_j(\bq)+\bigg(\min_{1\le j\le m}\frac{\psi_j(\bq)}{|\bq|}\bigg)^{s}\Bigg)<\infty.\]
		By Borel--Cantelli lemma and the $\limsup$ nature of $W(n,m;\Psi)$,
		we have
		\[\mu\big(W(n,m;\Psi)\big)=0.\]
		This contradicts that the support of $\mu$ is contained in $W(n,m;\Psi)$.
	\end{proof}
\begin{rem}
	The estimation of  $S_\bq$	is rather crude and can be refined as follows, under the assumption that $s$ is non-integral but not necessarily $s<1$. Let $k$ be the unique integer such that $k<s<k+1$. Using the fact that for any $t\ge 1$,
		\[\begin{split}
		&\#\{\bt\in\Z^m:|\bt|=t\text{ and $|t_j|\le 2/\delta_j$ for $1\le j\le m$}\}\\
		\asymp &\begin{cases}
			t^{m-1}&\text{if $t\le 2/\delta_1$},\\
			(\delta_1\cdots\delta_{j-1})^{-1}t^{m-j}&\text{if $2/\delta_{j-1}< t\le 2/\delta_{j}$ for some $j\ge 2$},
		\end{cases}
	\end{split},\]
	one can actually show that $S_\bq\asymp(\delta_1\cdots\delta_{m-k})^{-1}\delta_{m-k}^{s-k}$. Therefore,
	\[\mu\big(R(\bq,\bd)\big)\ll \delta_1\cdots\delta_m+O(|\bq|^{-s}\delta_{m-k+1}\cdots\delta_m\delta_{m-k}^{s-k}).\]
	It is fairly reasonable to conjecture that this estimation may correspond to the upper bound for $\fdim W(n,m;\Psi)$, and that the lower bound should coincide with this value as well. However, to the best of the author's knowledge, verifying this conjecture remains out of reach for most known methods unless $s<1$.
\end{rem}

	\subsection{Lower bound for $\fdim W(n,m;\Psi)$}
	The lower bound follows immediately from Theorem \ref{t:fdim}.

	\begin{proof}[Proof of Theorem \ref{t:fdimW}: lower bound]
		Let $\Psi=(\psi_1,\dots,\psi_m)$ be an $m$-tuple of multivariable approximating function. Define $\psi:\Z^n\to\R_{\ge 0}$ as follows:
		\[\psi(\bq)=\min_{1\le j\le m}\psi_j(\bq).\]
		Trivially, one has $W(n,m;\psi)\subset W(n,m;\Psi)$. Applying Theorem \ref{t:fdim}, we obtain the desired lower bound.
	\end{proof}

\section{Lebesgue measure of $\cm^\times(n,m;\psi)$}\label{s:Lebpartm}
 For $\delta>0$ and  $\bq=(q_1,\dots,q_n)\in\Z^n$, let
\[\cm(\bq,\delta):=\bigg\{\bx\in[0,1]^{nm}:\prod_{j=1}^m\|\bq\bx_j\|<\delta\bigg\}.\]
The proof of the convergence part of Theorem \ref{t:mulm} is almost identical to that of Theorem \ref{t:Schmidt}, with the only difference being the estimation of the Lebesgue measure of $\cm(\bq,\delta)$. To this end, in comparison with the well-known case ($n=1$), a useful technique is to map it into a lower-dimensional Euclidean space through a measure-preserving mapping.
\begin{lem}\label{l:meamc}
	Let $0<\delta<1/2$. We have
	\[\lm^{nm}\big(\cm(\bq,\delta)\big)\asymp \delta\log^{m-1}(\delta^{-1}).\]
\end{lem}
\begin{proof}
	Consider the transformation $T_\bq:[0,1]^{nm}\to[0,1]^m$ defined by
	\[T_\bq(\bx):=\bq\bx\mod 1.\]
	It is readily verified that
	\[\cm(\bq,\delta)=T_\bq^{-1}\bigg(\bigg\{(x_1,\dots,x_m)\in[0,1]^{m}:\prod_{j=1}^m\|x_j\|<\delta\bigg\}\bigg).\]
	The transformation $T_\bq$ is measure preserving, i.e. for any measurable set $A\subset [0,1]^m$, we have that $\lm^{nm}(T_\bq^{-1}A)=\lm^m(A)$; see \cite[Equation (48)]{Sp79book}. Therefore, we have
	\[\begin{split}
		\lm^{nm}\big(\cm(\bq,\delta)\big)&=\lm^m\bigg(\bigg\{(x_1,\dots,x_m)\in[0,1]^{m}:\prod_{j=1}^m\|x_j\|<\delta\bigg\}\bigg)\\
		&\asymp \delta\log^{m-1}(\delta^{-1}).\qedhere
	\end{split}\]
\end{proof}
By Lemma \ref{l:meamc}, the convergence part of Theorem \ref{t:Schmidt} follows the same lines as the proof of Lemma \ref{l:meaweic}.
\begin{lem}
	Let $n,m\ge 1$. Then,
	\[\sum_{\bq\in\Z^n}\psi(\bq)\log^{m-1}(\psi(\bq)^{-1})<\infty\quad\Longrightarrow\quad \lm^{nm}\big(\cm^\times(n,m;\psi)\big)=0.\]
\end{lem}
The divergence part of the theorem  relies on the following zero-one law.
\begin{thm}[{\cite[Theorem 4.1]{Li13}}]\label{l:reducem}
	For any $n,m\ge 1$, let $\psi:\Z^n\to\R_{\ge 0}$ be a multivariable approximating function. Then,
	\[\lm^{nm}\big(\cm^\times(n,m;\psi)\big)\in\{0,1\}.\]
	In particular,
	\[\lm^{nm}\big(\cm^\times(n,m;\psi)\big)>0\quad\Longrightarrow\quad \lm^{nm}\big(\cm^\times(n,m;\psi)\big)=1.\]
\end{thm}
For $0<\delta<2^{-m}$ and  $\bq=(q_1,\dots,q_n)\in\Z^n$, let
\[\begin{split}
	\cm'(\bq,\delta):=\bigg\{\bx\in[0,1]^{nm}:\prod_{j=1}^m|\bq\bx_j-p_j|<\delta &\text{ for some $\bp\in\Z^n$}\\
	&\text{ with $\gcd(\bq,p_j)=1$ for all $j$}\big\}.
\end{split}\]
The main ingredient in proving the quasi-independence result is to express each set $\cm'(\bq,\psi(\bq))$ as a finite union of sets of the form $R'(\bq,\bd)$ (see \eqref{eq:R'} for the definition)  for suitably chosen $\bd$. This expression enables the application of the quasi-independence estimates between sets $R'(\bq_1,\bd_1)$ and $R'(\bq_2,\bd_2)$, established in Lemma \ref{l:mea}, to this multiplicative setup. For any $0<\delta\le 2^{-m}$, let $N=N(\delta)\ge m$ be the unique integer such that $2^{-N-1}<\delta\le 2^{-N}$. Define
\begin{equation}\label{eq:amd}
	\ca_m(\delta)=\{\bk\in\Z^m:k_1,\dots,k_m\ge 0, k_1+\cdots+k_m=N-m\}.
\end{equation}
\begin{lem}\label{l:rectangledecom}
	Let $0<\delta\le 2^{-m}$ and $N=N(\delta)\ge m$ be defined as above. Then,
	\[\cm'(\bq,\delta)\subset \bigcup_{\bk\in\ca_m(\delta)}R'(\bq,2^{-\bk})\subset \cm'(\bq,2^{m+1}\delta),\]
	where $2^{-\bk}:=(2^{-k_1},\dots,2^{-k_m})$. A similar argument applies with $\cm'(\bq,\delta)$ and $R'(\bq,2^{-\bk})$ replaced by $\cm(\bq,\delta)$ and $R(\bq,2^{-\bk})$ (see \eqref{eq:rqdelta} for the definition), respectively.
\end{lem}
\begin{proof}
	We prove the right inclusion first. Suppose that $\bx\in R'(\bq,2^{-\bk})$ for some $\bk\in\ca_m(\delta)$. Then, there exists $\bp\in\Z^m$ with $\gcd(\bq,p_j)=1$ for $1\le j\le m$ such that
	\[|\bq\bx_j-p_j|<2^{-k_j}\quad\text{for $1\le j\le m$}.\]
	Multiplying over $j$, we obtain
	\[\prod_{j=1}^{m}|\bq\bx_j-p_j|<2^{-(k_1+\cdots+k_m)}\le 2^{-N+m}\le 2^{m+1}\delta,\]
	since $2^{-N-1}<\delta\le 2^{-N}$. This means that $\bx\in \cm'(\bq,2^{m+1}\delta)$.

	To prove the first inclusion, let $\bx\in M'(\bq,\delta)$. For each $j$, let $k_j$ be the largest integer such that
	\[2^{-k_j-1}\le|\bq\bx_j-p_j|<2^{-k_j}.\]
	Clearly,  $k_j\ge 0$ and
	\[\prod_{j=1}^{m}|\bq\bx_j-p_j|<\delta\quad\Longrightarrow\quad 2^{-(k_1+\cdots+k_m)-m}\le \delta\quad\Longrightarrow\quad k_1+\cdots+k_m\ge N-m.\]
	Thus, there exists $\bl\in \ca_m(\delta)$ such that $l_j\le k_j$ for all $j$ and $\bx\in R'(\bq,2^{-\bl})$.

	The argument for $\cm(\bq,\delta)$ and $R(\bq,2^{-\bk})$ follows in exactly the same manner.
\end{proof}
With the relation established above, we are now in a position to estimate the cardinality of $\ca_m(\delta)$.
\begin{lem}\label{l:caramd}
	For any $0<\delta\le2^{-m}$,
	\begin{equation}\label{eq:caramd}
		\#\ca_m(\delta)\asymp \log^{m-1}(\delta^{-1}),
	\end{equation}
	where the implied constant depends on $m$ only.
\end{lem}
\begin{proof}
	We proceed by induction. For $m=1$, $\ca_1(\delta)$ contains exactly one element, namely $N-1$. Thus, $\#\ca_1(\delta)=1$, which is equal to the right-hand side of \eqref{eq:caramd}.

	Suppose that \eqref{eq:caramd} holds for some $m\ge 1$. We will use the `slicing' technique to prove that it holds for $m+1$ as well. For any $\bk=(k_1,\dots,k_m,k)\in\ca_{m+1}(\delta)$, by definition, we have
	\[k_1+\cdots+k_m+k=N-(m+1)\quad
	\Longrightarrow\quad k_1+\cdots+k_m=(N-k-1)-m.\]
	It follows that
	\[\ca_{m+1}(\delta)= \bigcup_{k=0}^{N-(m+1)}\ca_m(2^{k+1}\delta).\]
	Since $k\le N-(m+1)$ and $2^{-N-1}<\delta\le 2^{-N}$, we have
	\[2^{k+1}\delta\le 2^{N-(m+1)+1}\delta\le 2^{-m}.\]
	By induction hypothesis and using the fact $N\asymp \log_2(\delta^{-1})\asymp \log(\delta^{-1})$, we have
	\[\#\ca_{m+1}(\delta)\asymp\sum_{k=0}^{N-(m+1)}\log^{m-1}(2^{-(k+1)}\delta^{-1})\ll N\log^{m-1}(\delta^{-1})\asymp\log^m(\delta^{-1})\]
	and
	\[\#\ca_{m+1}(\delta)\gg\sum_{k=0}^{N/2-(m+1)}\log^{m-1}(2^{-(k+1)}\delta^{-1})\gg N\log^{m-1}(\delta^{-1})\asymp\log^m(\delta^{-1}),\]
	where the second estimate follows from
	\[2^{-(k+1)}\delta^{-1}\ge 2^{-N/2}\cdot 2^N\asymp \delta^{-1/2}\quad\text{for $0\le k\le N/2-(m+1)$}.\qedhere\]
\end{proof}
The Lebesgue measure of the sets $\cm'(\bq,\delta)$ and their correlations can be estimated as follows.
\begin{lem}\label{l:quasimeam}
	Let $n, m\ge 1$. For any $\bq\in \Z^n\setminus\{0\}$ and $0<\delta\le2^{-m}$, we have
	\[\lm^{nm}\big(\cm'(\bq,\delta)\big)\gg \delta\bigg(\frac{\varphi(d)}{d}\bigg)^m\log^{m-1}(\delta^{-1}),\]
	where $d=\gcd(\bq)$. Moreover, for any $\bq_1\ne\pm\bq_2\in\Z^n\setminus\{0\}$ and any $0<\delta_1,\delta_2<2^{-m}$,
	\[\lm^{nm}\big(\cm'(\bq_1,\delta_1)\cap \cm'(\bq_2,\delta_2)\big)\ll\delta_1\log^{m-1}(\delta_1^{-1})\cdot \delta_2\log^{m-1}(\delta_2^{-1}).\]
\end{lem}
\begin{proof}
	Let $\bq\in\Z^n$ and $0<\delta\le 2^{-m}$. Write $\bq'=\bq/d$. Let $\bx\in \cm'(\bq,\delta)$. By definition, there are $\bp\in \Z^m$ such that $\prod_{j=1}^{m}|\bq\bx_j-p_j|<\delta$ and $\gcd(\bq,p_j)=1$ for $1\le j\le m$. Equivalently, we have that
	\[\prod_{j=1}^{m}|d\bq'\bx_j-p_j|<\delta,\quad\text{$\gcd(d,p_j)=1$ for $1\le j\le m$}.\]
	Consider the measure preserving transformation  $T_{\bq'}:[0,1]^{nm}\to[0,1]^m$ defined by
	\[T_{\bq'}(\bx):=\bq'\bx\mod 1.\]
	It is readily verified that
	\[\begin{split}
		\cm'(\bq,\delta)
		=T_{\bq' }^{-1}\bigg(\bigg\{(x_1,\dots,x_m)\in[0,1]^{m}:&\prod_{j=1}^m|dx_j-p_j|<\delta\text{ for some $\bp\in\Z^n$}\\
			&\text{ with $\gcd(d,p_j)=1$ for all $j$}\bigg\}\bigg).
	\end{split}\]
	The set inside the brackets on the right-hand side contains $\varphi(d)^m$ `star-shaped' sets  with centers at the rational points $\bp/d$; namely, each `star-shaped' set has the following form,
	\[d^{-1}\cdot\bigg\{(x_1,\dots,x_m)\in[0,1]^{m}:\prod_{j=1}^m|x_j|<\delta\bigg\}+\bp/d,\]
	where $\alpha A+\bz:=\{\alpha\by+\bz: \by\in A\}$ for $\alpha\in \R$ and $\bz\in\R^m$.
	 More importantly, each `star-shaped' set has $m$-dimensional Lebesgue  $\asymp d^{-m}\delta\log^{m-1}(\delta^{-1})$.
	This together with the measure-preserving property of $T_{\bq'}$ yields
	\[\lm^{nm}\big(\cm'(\bq,\delta)\big)\gg \varphi(d)^m\cdot  d^{-m}\delta\log^{m-1}(\delta^{-1}).\]

	Let $\bq_1\ne\pm \bq_2\in\Z^n\setminus\{0\}$ and $0<\delta_1,\delta_2<2^{-m}$. Recall the definition of $\ca_m(\delta)$ given in \eqref{eq:amd}. By Lemma \ref{l:rectangledecom}, we have
	\[\cm'(\bq_1,\delta_1)\cap \cm'(\bq_2,\delta_2)\subset \bigcup_{\bk\in\ca_m(\delta_1)}\bigcup_{\bl\in\ca_m(\delta_2)}\big(R'(\bq_1,2^{-\bk})\cap R'(\bq_2,2^{-\bl})\big).\]
	It follows from the quasi-independent result, established in Lemma \ref{l:mea}, that\[\lm^{nm}\big(R'(\bq_1,2^{-\bk})\cap R'(\bq_2,2^{-\bl})\big)\ll 2^{-(k_1+\cdots+k_m)-(l_1+\cdots+l_m)}\ll\delta_1\delta_2,\]
	since $j_1+\cdots+j_m=N(\delta)-m\asymp\log_2(\delta^{-1})-m$ for any $\textbf{j}\in\ca_m(\delta)$. Summing over $\bk$ and $\bl$, we deduce from Lemma \ref{l:caramd} that
	\[\begin{split}
	\lm^{nm}\big(\cm'(\bq_1,\delta_1)\cap \cm'(\bq_2,\delta_2)\big)&\ll\#\ca_m(\delta_1)\cdot\#\ca_m(\delta_2)\cdot\delta_1\delta_2\\
	&\ll \delta_1\log^{m-1}(\delta_1^{-1})\cdot \delta_2\log^{m-1}(\delta_2^{-1}).\qedhere
\end{split}\]
\end{proof}
Instead of providing the complete proof, we only outline the main idea. The details are almost the same as those in the proof of Theorem \ref{t:Schmidt}.
\begin{proof}[Sketch proof of Theorem \ref{t:mulm}: divergence part]
	By Lemma \ref{l:positivedensity}, there exists a set $\Lambda\subset \N$ with positive density such that $\varphi(q)/q\ge 1/2$ for all $q\in\Lambda$. Therefore, for any $\bq\in\Z^n$ with $|\bq|=q$, by Lemmas \ref{l:meamc} and \ref{l:quasimeam},
	\[\lm^{nm}\big(\cm'(\bq,\delta)\big)\asymp \delta\log^{m-1}(\delta^{-1}).\]
	Since the quasi-independent result established in Lemma \ref{l:quasimeam} does not apply to the case $\bq_1=\pm\bq_2$, for any $q\in\N$ we define
	\begin{equation}\label{eq:deltaq}
		\Delta(q):=\{\bq\in\Z^n:|\bq|=q\text{ and $\psi(\bq)\ge \psi(-\bq)$}\}.
	\end{equation}
	The monotonicity assumption on $\psi$ ensures that
		\[\sum_{\bq\in\Z^n}\psi(\bq)\log^{m-1}(\psi(\bq)^{-1})=\infty\quad\Longrightarrow\quad\sum_{q\in\Lambda}\sum_{\bq\in\Delta(q)}\psi(\bq)\log^{m-1}(\psi(\bq)^{-1})=\infty.\]
		By Lemma \ref{l:quasimeam}, for any $\bq_1\in\Delta(|\bq_1|)$ and $\bq_2\in\Delta(|\bq_2|)$ with $|\bq_1|,|\bq_2|\in\Lambda$, we have
		\[\begin{split}
			&\lm^{nm}\big(\cm'\big(\bq_1,\psi(\bq_1)\big)\cap \cm'\big(\bq_2,\psi(\bq_2)\big)\big)\\
			\ll&\delta_1\log^{m-1}(\delta_1^{-1})\cdot \delta_2\log^{m-1}(\delta_2^{-1})\asymp \lm^{nm}\big(\cm'\big(\bq_1,\psi(\bq_1)\big)\big) \lm^{nm}\big(\cm'\big(\bq_2,\psi(\bq_2)\big)\big).
		\end{split}\]
		By Lemma \ref{l:quasiind} we have that $\cm^\times(n,m;\psi)$ has positive $nm$-dimensional Lebesgue. By the zero-one law in Lemma \ref{l:reducem}, $\cm^\times(n,m;\psi)$ must carry full Lebesgue measure.
\end{proof}
\section{Hausdorff measure of $\cm^\times(n,m;\psi)$}\label{s:Hausdorffpartm}

 For $\delta>0$ and  $\bq=(q_1,\dots,q_n)\in\Z^n$, let
\[\cm(\bq,\delta):=\bigg\{\bx\in[0,1]^{nm}:\prod_{j=1}^m\|\bq\bx_j\|<\delta\bigg\}\]
be defined as in Section \ref{s:Lebpartm}.
\subsection{Convergence parts of Theorems \ref{t:mulh} and \ref{t:mulhs}}\label{ss:convergencem}
Since the convergence parts in Theorems \ref{t:mulh} and \ref{t:mulhs} follow directly by considering the natural cover of $\cm(\bq,\delta)$, we will establish them simultaneously in this subsection.
\begin{lem}\label{l:mulplicativecover}
	Let $\bq\in\Z^n\setminus\{0\}$ and  $0<\delta\le2^{-m}$. Suppose that $f\preceq nm$. Then, there exists a finite collection $\cb$ of balls that covers $\cm(\bq,\delta)$ and satisfies
	\[|B|\ll \frac{1}{|\bq|}\text{ for all $B\in\cb$}\qaq \sum_{B\in\cb}f(|B|)\ll   f\bigg(\frac{\delta}{|\bq|}\bigg)\bigg(\frac{\delta}{|\bq|}\bigg)^{1-nm} |\bq|\log^{m-1}(\delta^{-1}).\]
	Suppose further that  $f\preceq(nm-1+s)$ for some $0<s<1$. Then the $\log^{m-1}$ term in the above inequality can be removed.
\end{lem}
\begin{proof}
	By Lemma \ref{l:rectangledecom},
	\[\cm(\bq,\delta)\subset \bigcup_{\bk\in\ca_m(\delta)}R(\bq,2^{-\bk}),\]
	where recall that $2^{-\bk}:=(2^{-k_1},\dots,2^{-k_m})$,
	\[\ca_m(\delta):=\{\bk\in\Z^m:k_1,\dots,k_m\ge 0, k_1+\cdots+k_m=N-m\}\]
	and $N=N(\delta)\ge m$ is the unique integer such that $2^{-N-1}<\delta\le 2^{-N}$. By symmetry, $\ca_m(\delta)$ can be decompose into $m$ sets $\ca_m(\delta;i)$ ($1\le i\le m$) with the property that for any $\bk\in\ca_m(\delta;i)$, $k_i$ is the maximum value among $k_1,\dots,k_m$. In the rest of the proof, for ease of discussion, by symmetry we only consider $\ca_m(\delta;m)$. For each $\bk\in\ca_m(\delta;m)$, we will cover $R(\bq,2^{-\bk})$ by balls with radius $2^{-k_m}/|\bq|$. In view of the discussion given in Section \ref{ss:convergence}, the $f$-volume cover of $R(\bq,2^{-\bk})$ is majorized by
	\begin{align}
		&\ll |\bq|^m\cdot f\bigg(\frac{2^{-k_m}}{|\bq|}\bigg)\bigg(\frac{2^{-k_m}}{|\bq|}\bigg)^{1-n}\prod_{j=1}^{m-1}\bigg(\frac{2^{-{k_j}}}{|\bq|}\bigg(\frac{2^{-k_m}}{|\bq|}\bigg)^{-n}\bigg)\notag\\
		&=  f\bigg(\frac{2^{-k_m}}{|\bq|}\bigg)\bigg(\frac{2^{-k_m}}{|\bq|}\bigg)^{1-nm} |\bq|\cdot2^{-(k_1+\cdots+k_{m-1})}\notag\\
		&\asymp f\bigg(\frac{2^{-k_m}}{|\bq|}\bigg)\bigg(\frac{2^{-k_m}}{|\bq|}\bigg)^{1-nm}|\bq|\cdot 2^{-N+k_m},\label{eq:f2km}
	\end{align}
	since $k_1+\cdots+k_m=N-m$. Suppose that $f\preceq nm$. Since $k_m\le N-m<N$,
	\[f\bigg(\frac{2^{-k_m}}{|\bq|}\bigg)\bigg(\frac{2^{-k_m}}{|\bq|}\bigg)^{-nm}\le f\bigg(\frac{2^{-N}}{|\bq|}\bigg)\bigg(\frac{2^{-N}}{|\bq|}\bigg)^{-nm}.\]
	Substitute this inequality,  \eqref{eq:f2km} is majorized by
	\[\ll f\bigg(\frac{2^{-N}}{|\bq|}\bigg)\bigg(\frac{2^{-N}}{|\bq|}\bigg)^{1-nm}|\bq|\asymp f\bigg(\frac{\delta}{|\bq|}\bigg)\bigg(\frac{\delta}{|\bq|}\bigg)^{1-nm} |\bq|,\]
	where we use $2^{-N-1}< \delta\le 2^{-N}$.
	The first point of the lemma follows from $\#\ca_m(\delta;m)\le\#\ca_m(\delta)\asymp\log^{m-1}(\delta^{-1})$ (see Lemma \ref{l:caramd}).

Suppose now that $f\preceq (nm-1+s)$ for some $0<s<1$. Then,
	\[f\bigg(\frac{2^{-k_m}}{|\bq|}\bigg)\bigg(\frac{2^{-k_m}}{|\bq|}\bigg)^{-(nm-1+s)}\le f\bigg(\frac{2^{-N}}{|\bq|}\bigg)\bigg(\frac{2^{-N}}{|\bq|}\bigg)^{-(nm-1+s)}.\]
	Therefore, the estimate of \eqref{eq:f2km} can be improved as
	\[\ll  f\bigg(\frac{2^{-N}}{|\bq|}\bigg)\bigg(\frac{2^{-N}}{|\bq|}\bigg)^{1-nm}|\bq|\cdot 2^{(N-k_m)(s-1)}.\]
	By Lemma \ref{l:caramd}, we have
	\begin{align}
		\sum_{\bk\in\ca_m(\delta;m)} 2^{(N-k_m)(s-1)}&\ll \sum_{k_m=0}^{N-m}\big(\#\ca_{m-1}(2^{k_m+1}\delta)\big) 2^{(N-k_m)(s-1)}\notag\\
		&\ll\sum_{k_m=0}^{N-m}(N-k_m)^{m-2}\cdot 2^{(N-k_m)(s-1)}\ll 1,\label{eq:logterm}
	\end{align}
	since $s<1$. This means that the $f$-volume cover of $\cm(\bq,\delta)$ does not exceed
	\[\ll   f\bigg(\frac{2^{-N}}{|\bq|}\bigg)\bigg(\frac{2^{-N}}{|\bq|}\bigg)^{1-nm}|\bq| \asymp  f\bigg(\frac{\delta}{|\bq|}\bigg)\bigg(\frac{\delta}{|\bq|}\bigg)^{1-nm}|\bq|.\qedhere\]
\end{proof}
As a consequence, we have the following two results.
\begin{proof}[Proof of Theorem \ref{t:mulh}: convergence part]
	Assume that $f\preceq (nm-1+s)$ for some $0<s<1$, and that
	\begin{equation}\label{eq:conconm}
		\sum_{\bq\in\Z^n\setminus\{0\}}f\bigg(\frac{\psi(\bq)}{|\bq|}\bigg)\bigg(\frac{\psi(\bq)}{|\bq|}\bigg)^{1-nm}|\bq|<\infty.
	\end{equation}
	By Lemma \ref{l:mulplicativecover} and the convergence of the series in \eqref{eq:conconm},
	\[\hm^f\big(\cm^\times(n,m;\psi)\big)\le  \liminf_{q\to\infty}\sum_{\bq\in\Z^n:|\bq|\ge q}f\bigg(\frac{\psi(\bq)}{|\bq|}\bigg)\bigg(\frac{\psi(\bq)}{|\bq|}\bigg)^{1-nm}|\bq|=0.\qedhere\]
\end{proof}
\begin{proof}[Proof of Theorem \ref{t:mulhs}: convergence part]
	Assume that $f\preceq nm$, and that
	\begin{equation}\label{eq:conconms}
		\sum_{\bq\in\Z^n\setminus\{0\}}f\bigg(\frac{\psi(\bq)}{|\bq|}\bigg)\bigg(\frac{\psi(\bq)}{|\bq|}\bigg)^{1-nm}|\bq|\log^{m-1}(\psi(\bq)^{-1})<\infty.
	\end{equation}
	By Lemma \ref{l:mulplicativecover} and the convergence of the series in \eqref{eq:conconms},
	\[\hm^f\big(\cm^\times(n,m;\psi)\big)\le  \liminf_{q\to\infty}\sum_{\bq\in\Z^n:|\bq|\ge q}f\bigg(\frac{\psi(\bq)}{|\bq|}\bigg)\bigg(\frac{\psi(\bq)}{|\bq|}\bigg)^{1-nm}|\bq|\log^{m-1}(\psi(\bq)^{-1})=0.\qedhere\]
\end{proof}
\subsection{Divergence part of Theorem \ref{t:mulh}}\label{ss:divergencem}
Assume that
\begin{equation}\label{eq:multiplihaudiv}
	\sum_{\bq\in\Z^n\setminus\{0\}}f\bigg(\frac{\psi(\bq)}{|\bq|}\bigg)\bigg(\frac{\psi(\bq)}{|\bq|}\bigg)^{1-nm}|\bq|=\infty.
\end{equation}
Since the sum does not carry any $\log$ term, the Hausdorff measure result for the weighted case $W(n,m;\Psi)$ is applicable.
Let $\Psi=(\psi_1,\dots,\psi_m)$ be an $m$-tuple of multivariable functions defined by $\psi_1=\cdots=\psi_{m-1}=1$ and $\psi_m=\psi$.
It is easy to see that $W(n,m;\Psi)\subset \cm^\times(n,m;\psi)$. Moreover, the monotonicity of $\psi$ implies so is $\Psi$. For such $\Psi$,
\[\begin{split}
	t_{\bq}(\Psi,f)=&\min\bigg\{f(|\bq|^{-1})|\bq|^{-(1-n)m},f\bigg(\frac{\psi(\bq)}{|\bq|}\bigg)\bigg(\frac{\psi(\bq)}{|\bq|}\bigg)^{(1-n)m}\psi(\bq)^{1-m}\bigg\}\\
	=& \min\bigg\{f(|\bq|^{-1})|\bq|^{1-nm},f\bigg(\frac{\psi(\bq)}{|\bq|}\bigg)\bigg(\frac{\psi(\bq)}{|\bq|}\bigg)^{1-nm}\bigg\}\cdot |\bq|^{1-m}\\
	=&f\bigg(\frac{\psi(\bq)}{|\bq|}\bigg)\bigg(\frac{\psi(\bq)}{|\bq|}\bigg)^{1-nm}|\bq|^{1-m},
\end{split}\]
since $(nm-1)\preceq f$. The divergence of the series in \eqref{eq:multiplihaudiv} implies that
\[\sum_{\bq\in\Z^n\setminus\{0\}}	t_{\bq}(\Psi,f)|\bq|^m=\infty.\]
Thus, by Theorem \ref{t:main},
\[\hm^f\big(\cm^\times(n,m;\psi)\big)\ge \hm^f\big(W(n,m;\Psi)\big)=\hm^f([0,1]^{nm}).\]
\subsection{Divergence part of Theorem \ref{t:mulhs}}
Let $(nm-1)\preceq f\prec nm$ be a dimension function such that for any $0<t<1$ there exists $r_0=r_0(t)$ for which
\begin{equation}\label{eq:fasymprnm}
	\frac{f(\alpha r)}{f(r)}\asymp \alpha^{nm},\quad \text{for all $1<\alpha<r^{-t}$ and $\alpha r<r_0$},
\end{equation}
where the implied constant may depend on $t$. We will prove that $\cm^\times(n,m;\psi)$ has full Hausdorff $f$-measure under the additional assumption:
\begin{equation}\label{eq:multiplihaudivs}
	\sum_{\bq\in\Z^n\setminus\{0\}}f\bigg(\frac{\psi(\bq)}{|\bq|}\bigg)\bigg(\frac{\psi(\bq)}{|\bq|}\bigg)^{1-nm}|\bq|\log^{m-1}(\psi(\bq)^{-1})=\infty.
\end{equation}
Let us begin by outlining our strategy:
\begin{enumerate}[\upshape (1)] \item First, we use the divergence of the series in \eqref{eq:multiplihaudivs} to define a new function $\phi$ such that $\phi\ge\psi$ and $\cm^\times(n,m;\phi)$ has full $nm$-dimensional Lebesgue measure.
\item Second, we apply Lemma \ref{l:rectangledecom} to rephrase $\cm^\times(n,m;\phi)$ as a $\limsup$ set defined by balls.
\item Finally, we estimate the Hausdorff $f$-content of the set $B \cap \cm(\bq, \psi(\bq))$ for a suitably chosen ball $B$ determined in the previous step, where the assumption \eqref{eq:fasymprnm} plays a crucial role. Combined with the `balls-to-rectangles' mass transference principle (see Theorem \ref{t:weaken}), this completes the proof. \end{enumerate}

To avoid some trivial discussions, we may assume that
\begin{equation}\label{eq:sumpsicon}
	\sum_{\bq\in\Z^n\setminus\{0\}}\psi(\bq)\log^{m-1}(\psi(\bq)^{-1})<\infty.
\end{equation}
Otherwise, by Theorem \ref{t:mulm}, we have $\hm^{nm}(\cm^\times(n,m;\psi))=1$, and thus the proof is complete.

The convergence of the series in \eqref{eq:sumpsicon} along with the monotonicity of $\psi$ implies that for all $\bq\in\Z^n$ with $|\bq|$ large enough,
\begin{equation}\label{eq:psiq<q-1}
	\psi(\bq)\le |\bq|^{-1}.
\end{equation}
We claim that for any dimension function $f$ satisfying \eqref{eq:fasymprnm}, it holds that for any $s < nm$,
\[f(r)\le r^{s}\quad \text{whenever $r$ is sufficiently small.}\]
 Seeking a contradiction, suppose that there exists infinitely many $r$ such that $f(r)>r^{s}$. Applying \eqref{eq:fasymprnm} with $t=s/(nm)$ and $\alpha=r^{-t}$, we get
\[f(\alpha r)=f(r^{1-t})\asymp r^{-tnm}f(r)\ge r^{-s}r^s=1,\]
which contradicts that $f(r)\to 0$ as $r\to 0^+$. Write
\[\phi(\bq)=f\bigg(\frac{\psi(\bq)}{|\bq|}\bigg)\bigg(\frac{\psi(\bq)}{|\bq|}\bigg)^{1-nm}|\bq|.\]
Let $s=nm-1/4$. By the claim and \eqref{eq:psiq<q-1}, for any $|\bq|$ large enough,
\[\begin{split}
	&f\bigg(\frac{\psi(\bq)}{|\bq|}\bigg)\le \bigg(\frac{\psi(\bq)}{|\bq|}\bigg)^{nm-1/4}\\
	\Longrightarrow\quad&\phi(\bq)=f\bigg(\frac{\psi(\bq)}{|\bq|}\bigg)\bigg(\frac{\psi(\bq)}{|\bq|}\bigg)^{1-nm}|\bq|\le \psi(\bq)^{3/4}|\bq|^{1/4}\ll \psi(\bq)^{1/2}.
\end{split}\]
On the other hand, by $f\prec nm$,
\begin{equation}\label{eq:phiq>psiq}
	\phi(\bq)=f\bigg(\frac{\psi(\bq)}{|\bq|}\bigg)\bigg(\frac{\psi(\bq)}{|\bq|}\bigg)^{1-nm}|\bq|\ge\psi(\bq).
\end{equation}
Therefore, \eqref{eq:multiplihaudivs} implies that
\[\sum_{\bq\in\Z^n\setminus\{0\}}\phi(\bq)\log^{m-1}(\phi(\bq)^{-1})=\infty.\]
By Theorem \ref{t:mulm}, $\cm^\times(n,m;\phi)$ has full $nm$-dimensional Lebesgue measure. Recall the definition of $\ca_m(\delta)$ in \eqref{eq:amd}. By Lemma \ref{l:rectangledecom},
\[\cm^\times(n,m;\phi)=\limsup_{\bq\in\Z^n:|\bq|\to\infty}\cm\big(\bq,\phi(\bq)\big)\subset \limsup_{\bq\in\Z^n:|\bq|\to\infty}\bigcup_{\bk\in\ca_m(\phi(\bq))}R(\bq,2^{-\bk}).\]
For each set $R(\bq,2^{-\bk})$, we can cover it by a finite collection $\mathcal C_\phi(\bq,\bk)$ of balls with centers at $R(\bq,2^{-\bk})$ and radius $2^{m-k_{\max}}/|\bq|:=2^{m-\max_{1\le i\le m}k_i}/|\bq|$. Therefore, the $\limsup$ set
\[\limsup_{\bq\in\Z^n:|\bq|\to\infty}\bigcup_{\bk\in\ca_m(\phi(\bq))} \bigcup_{B\in\mathcal{C}_\phi(\bq,\bk)}B.\]
defined by balls has full $nm$-dimensional Lebesgue measure. Observe that
\begin{equation}\label{eq:Bcapmul}
	\limsup_{\bq\in\Z^n:|\bq|\to\infty}\bigcup_{\bk\in\ca_m(\phi(\bq))} \bigcup_{B\in\mathcal{C}_\phi(\bq,\bk)}B\cap \cm\big(\bq,\psi(\bq)\big)\subset \cm^\times(n,m;\psi).
\end{equation}
The full Hausdorff $f$-measure statement will hold for $\cm^\times(n,m;\psi)$ once we prove the following result.
\begin{lem}
	Let $B$ be as in \eqref{eq:Bcapmul}. Suppose that \eqref{eq:fasymprnm} holds. Then,
	\[\hc^f\big(B\cap \cm\big(\bq,\psi(\bq)\big)\big)\gg |B|^{nm}.\]
	Consequently, $\hm^f(\cm^\times(n,m;\psi))=\hm^f([0,1]^{nm})$.
\end{lem}
\begin{proof}
	Assume that $B$ is obtained from $R(\bq,2^{-\bk})$ for some $\bk\in \ca_m(\phi(\bq))$. Assume without loss of generality that $k_m=k_{\max}=\max_{1\le i\le m}k_i$. So the radius of $B$ is $2^{m-k_{\max}}/|\bq|=2^{m-k_m}/|\bq|$. We first claim that $B\cap \cm(\bq,\psi(\bq))$ contains a hyperrectangle, denoted by $R$, with side lengths
	\begin{equation}\label{eq:sidemul}
		\underbrace{\frac{2^{-k_m}}{|\bq|},\dots,\frac{2^{-k_m}}{|\bq|}}_{nm-1},\frac{2^{k_1+\cdots+k_{m-1}}\cdot\psi(\bq)}{|\bq|}.
	\end{equation}
	In the spirit of Lemma \ref{l:rectangledecom}, it is not difficult to verify that for $\bl\in\Z^m$ with $l_j=k_j$ ($1\le j\le m-1$) and $l_m=\log_2(\psi(\bq)^{-1})-l_1-\cdots -l_{m-1}$,
	\[R(\bq,2^{-\bl})\subset \cm\big(\bq,\psi(\bq)\big).\]
	Note that $B$ centered at $R(\bq,2^{-\bk})$ and
	\[R(\bq,2^{-\bl})=\prod_{j=1}^{m}\big\{\bx_j\in[0,1]^n:\|\bq\bx_j\|<2^{-l_j}\big\}.\]
We can see that $B\cap R(\bq,2^{-\bl})$ contains a hyperrectangle with side lengths described in \eqref{eq:sidemul}, and so the claim follows.

	Since $\psi(\bq)\le \phi(\bq)$ (see \eqref{eq:phiq>psiq}), for any $\bk\in \ca_m(\phi(\bq))$, by definition,
	\[2^{-k_m}\ge 2^{k_1+\cdots+k_{m-1}+m}\cdot \phi(\bq)> 2^{k_1+\cdots+k_{m-1}}\cdot\psi(\bq).\]
	By the monotonicity of Hausdorff content and apply Proposition \ref{p:rec} with $(nm-1)\preceq f\prec nm$,
	\[\begin{split}
		&\hc^f\big(B\cap \cm\big(\bq,\psi(\bq)\big)\big)
		\ge\hc^f(R)\\
	\asymp& \bigg(\frac{2^{-km}}{|\bq|}\bigg)^{nm-1}\bigg(\frac{2^{k_1+\cdots+k_{m-1}}\cdot\psi(\bq)}{|\bq|}\bigg)^{1-nm}f\bigg(\frac{2^{k_1+\cdots+k_{m-1}}\cdot\psi(\bq)}{|\bq|}\bigg).
	\end{split}\]
	By the definition of $\ca_m(\phi(\bq))$ and using $k_m=\max_{1\le i\le m}k_i$, we have
	\[\begin{split}
		k_1+\cdots+k_{m-1}\le \frac{(m-1)(\log_2(\phi(\bq)^{-1})-m)}{m}\ll \frac{(m-1)\log_2(\psi(\bq)^{-1})}{m}.
	\end{split}\]
	Therefore, \eqref{eq:fasymprnm} is applicable. By that formula,
	\[f\bigg(\frac{2^{k_1+\cdots+k_{m-1}}\cdot\psi(\bq)}{|\bq|}\bigg)\asymp f\bigg(\frac{\psi(\bq)}{|\bq|}\bigg)\cdot 2^{(k_1+\cdots+k_{m-1})nm}.\]
	It follows from the definition of $\phi(\bq)$ that
	\[\begin{split}
		\hc^f(R)
		&\asymp \bigg(\frac{2^{-km}}{|\bq|}\bigg)^{nm-1}2^{k_1+\cdots+k_{m-1}}\bigg(\frac{\psi(\bq)}{|\bq|}\bigg)^{1-nm}f\bigg(\frac{\psi(\bq)}{|\bq|}\bigg)\\
		&\asymp \bigg(\frac{2^{-km}}{|\bq|}\bigg)^{nm-1}2^{-k_m}\phi(\bq)^{-1}\cdot \bigg(\frac{\psi(\bq)}{|\bq|}\bigg)^{1-nm}f\bigg(\frac{\psi(\bq)}{|\bq|}\bigg)\\
		&=\bigg(\frac{2^{-km}}{|\bq|}\bigg)^{nm}\asymp|B|^{nm}.
	\end{split}\]
	The last conclusion of the lemma follows from Theorem \ref{t:weaken} directly.
\end{proof}
 \section{Fourier dimension of $\cm^\times(n,m;\psi)$}\label{s:Fourierpartm}
 Assume that $\sum_{\bq\in\Z^n\setminus\{0\}}^\infty\psi(\bq)\log^{m-1}(\psi(\bq)^{-1})<\infty$. Recall that
 \[\tau(\psi):=\inf\biggl\{\tau\ge 0:\sum_{\bq\in\Z^n\setminus\{0\}}\bigg(\frac{\psi(\bq)^{1/m}}{\bq}\bigg)^{\tau}<\infty\biggr\}.\]
 Note that no monotonicity assumption was posed on $\psi$.
 \subsection{Upper bound for $\fdim \cm^\times(n,m;\psi)$}
Assume for the contrary that there exists a measure $\mu$ supported on $\cm^\times(n,m;\psi)$ and $\tau>\tau(\psi)$ such that $|\widehat{\mu}(\bx)|\ll|\bx|^{-\tau}$ for all $\bx\in \R^{nm}$. By decreasing $\tau$ a little bit if necessary, we may assume that $\tau$ is not an integer. The proof of the upper bound will follow the same strategy as that of $\fdim W(n,m;\Psi)$. To this end, we need to estimate the $\mu$-measure of $\cm(\bq,\delta)$, where \[\cm(\bq,\delta):=\bigg\{\bx\in[0,1]^{nm}:\prod_{j=1}^m\|\bq\bx_j\|<\delta\bigg\}\]
is defined in the same way as in Section \ref{s:Lebpartm}.
 \begin{lem}\label{l:fdimmeauppmul}
 	Let $\bq\in \Z^n\setminus\{0\}$ and $0<\delta\le2^{-m}$. We have
 	\[\mu\big(\cm(\bq,\delta)\big)\ll (\delta+q^{-\tau}\delta^{\tau/m})\log^{m-1}(\delta^{-1}).\]
 \end{lem}
 \begin{proof}
 	By Lemma \ref{l:rectangledecom},
 		\[\cm(\bq,\delta)\subset \bigcup_{\bk\in\ca_m(\delta)}R(\bq,2^{-\bk}).\]
 		For each $\bk\in \ca_m(\delta)$, apply Lemma \ref{l:meaupp} with  $\bd=2^{-\bk}=(2^{-k_1},\dots,2^{-k_m})$ yields
 	\[\mu\big(R(\bq,2^{-\bk})\big)\ll \delta\Bigg(1+\sum_{\substack{\bt\in\Z^m\setminus\{0\}\\ \forall1\le j\le m, |t_j|\le 2^{k_j+1}}}|\widehat{\mu}(t_1\bq,\dots,t_m\bq)|\Bigg),\]
 	where we use $2^{-(k_1+\dots+k_m)}\asymp \delta$.
 	Let $S$ denote the inner-most summation. By the Fourier decay assumption,
 	\[S\ll |\bq|^{-\tau}\sum_{\substack{\bt\in\Z^m\setminus\{0\}\\ \forall1\le j\le m, |t_j|\le 2^{k_j+1}}}|\bt|^{-\tau}.\]
 	With the convention $1/0=\infty$, for any $\bt\in\Z^m\setminus\{0\}$,
 	\[\begin{split}
 		|\bt|^{-\tau}&=\min\{1,|t_1|^{-\tau},\dots,|t_d|^{-\tau}\}\\
 		&=\min\big\{\min\{1,|t_1|^{-\tau}\},\dots,\min\{1,|t_d|^{-\tau}\}\big\}\\
 		&\le \min\{1,|t_1|^{-\tau/m}\}\cdots\min\{1,|t_d|^{-\tau/m}\}.
 	\end{split}\]
 	Substitute this upper estimate, it follows that
 	\begin{align}
 		S&\ll |\bq|^{-\tau}\sum_{\substack{\bt\in\Z^m\setminus\{0\}\\ \forall1\le j\le m, |t_j|\le 2^{k_j+1}}}\min\{1,|t_1|^{-\tau/m}\}\cdots\min\{1,|t_m|^{-\tau/m}\}\notag\\
 		&\ll |\bq|^{-\tau}\prod_{j=1}^{m}\bigg(\sum_{0\le|t_j|\le2^{k_j+1}}\min\{1,|t_j|^{-\tau/m}\}\bigg).\label{eq:sfourierdecay}
 	\end{align}
 	Since $\tau$ is not an integer, we have $\tau/m\ne 1$. For each $1\le j\le m$,
 	\[\sum_{0\le|t_j|\le2^{k_j+1}}\min\{1,|t_j|^{-\tau/m}\}\ll \max\{1,2^{k_j(1-\tau/m)}\}.\]
 	Therefore,
 	\[\begin{split}
 		\mu\big(R(\bq,2^{-\bk})\big)&\ll \delta(1+S)\ll \delta\big(1+|\bq|^{-\tau}\max\{1,2^{(k_1+\cdots+k_m)(1-\tau/m)}\}\big)\\
 		&\ll \delta+|\bq|^{-\tau}\delta^{\tau/m}.
 	\end{split}\]
 	Summing over $\bk\in\ca_m(\delta)$ and using Lemma \ref{l:caramd} ($\#\ca_m(\delta)\asymp \log^{m-1}(\delta^{-1})$), we conclude the lemma.
 \end{proof}

 \begin{proof}[Proof of Theorem \ref{t:fdimM}: upper bound]
 	Let $\mu$ and $\tau>\tau(\psi)$ be as in the assumption at the begining of this subsection. For any $\bq\in\Z^n\setminus\{0\}$ with $\psi(\bq)\ne0$, by Lemma \ref{l:fdimmeauppmul},
 	\[\begin{split}
 		\mu\big(\cm(\bq,\psi(\bq))\big)&\ll (\psi(\bq)+|\bq|^{-\tau}\psi(\bq)^{\tau/m})\log^{m-1}(\psi(\bq)^{-1})\\
 		&\ll \psi(\bq)\log^{m-1}(\psi(\bq)^{-1})+\bigg(\frac{\psi(\bq)^{1/m}}{\bq}\bigg)^{\tau-\epsilon},
 	\end{split}\]
 	for some $\epsilon>0$ satisfying $\tau-\epsilon>\tau(\psi)$.
 	Summing over $\bq\in\Z^n\setminus\{0\}$, we have $\sum_{\bq\in\Z^n\setminus\{0\}}^\infty\mu(\cm(\bq,\psi(\bq)))<\infty$. By Borel--Cantelli lemma,
 	\[\mu\big(\cm^\times(n,m;\psi)\big)=\mu\Big(\limsup_{|\bq|\to\infty}\cm\big(\bq,\psi(\bq)\big)\Big)=0.\]
 	This contradicts that $\mu$ is supported on $\cm^\times(n,m;\psi)$. Therefore, by the arbitrariness of $\tau>\tau(\psi)$,
 	\[\fdim \cm^\times(n,m;\psi)\le 2\tau(\psi).\qedhere\]
 \end{proof}
 \subsection{Lower bound for $\fdim \cm^\times(n,m;\psi)$}
 \begin{proof}[Proof of Theorem \ref{t:fdimM}: lower bound]
 	It is easy to see that $W(n,m;\psi^{1/m})\subset \cm^\times(n,m;\psi)$. Therefore, by Theorem \ref{t:fdim},
 	\[\fdim \cm^\times(n,m;\psi)\ge \fdim W(n,m;\psi^{1/m})\ge 2s(\psi^{1/m})=2\tau(\psi).\qedhere\]
 \end{proof}

\section{Fourier dimensions of product sets and its application}\label{s:product}
We will show that the Fourier dimensions of product sets with null Lebesgue measure never increases.
\begin{proof}[Proof of Theorem \ref{t:product}]
	Let $E\subset\R^k$ and $F\subset\R^{d-k}$ be two Borel sets with zero Lebesgue measure. Seeking a contradiction, suppose that
	\[\fdim(E\times F)>\min\{\fdim E,\fdim F\}.\]
	Without loss of generality, assume that $\fdim E$ is the minimum value of the right. By definition, there exist $s>\fdim E$ and a probability measure $\mu$ with $\supp(\mu)\subset E\times F$ such that
	\[|\widehat{\mu}(\bx)|\ll |\bx|^{-s/2}\quad\text{for all $\bx\in\R^d$}.\]
	Let $\nu=\Pi(\mu)$ denote the projection of $\mu$ onto $E$. Since $\Pi$ is continuous, we have that $\Pi(\supp(\mu))$ is still compact. Moreover, by compactness we have
	\[\supp(\nu)= \Pi\big(\supp(\mu)\big)\subset E,\]
	which means that $\nu$ is a Borel probability measure supported on $E$. For any $\bx\in\R^k$ it follows that
	\begin{align}
		|\widehat{\nu}(\bx)|&=\bigg|\int_{\R^k}e^{-2\pi\ic\bm\xi_1\cdot\bx}\dif\nu(\bm\xi_1)\bigg|=\bigg|\int_{\R^d}e^{-2\pi\ic\bm\xi_1\cdot\bx}\dif\mu(\bm\xi_1,\bm\xi_2)\bigg|\notag\\
		&=|\widehat{\mu}(\bx,0)|\ll |\bx|^{-s/2}.\label{eq:decay}
	\end{align}
	We consider two cases.

	\noindent\textbf{Case 1:} $\fdim E<k$. The assumption $s>\fdim E$, together with \eqref{eq:decay}, contradicts the definition of $\fdim E$.

	\noindent\textbf{Case 2:} $\fdim E=k$.
	 If $s>k$, then $\widehat{\nu}\in L^2(\R^k)$, and by \cite[Theorem 3.3]{Ma15}, it follows that $\nu\in L^2(\R^k)$. This implies that
	 $\nu$ is absolutely continuous with respect to Lebesgue measure and supported on $E$, contradicting the fact that $E$ has zero Lebesgue measure.
	 Therefore, we must have $s\le k$, which in turn contradicts the assumption that $s>\fdim E=k$.
\end{proof}
\begin{rem}
	To ensure that the inequality $\fdim E\le \hdim E$ is always meaningful, we require $0\le s\le k$ in the definition of the Fourier dimension. However, if $E$ has positive Lebesgue measure, then it is possible for $E$ to support a measure whose Fourier decays with an exponent strictly greater than $k/2$. For example, let $E=[0,1]^k$ and let $\mu$ be the Lebesgue measure restricted to $E$. To exclude this possibility, we restrict $E$ to be a null set.
\end{rem}

\begin{proof}[Proof of Corollary \ref{c:prod}]
	Let $W(n_i,m_i;\Psi_i)$ be the set with smallest Fourier dimension among $1\le \ell\le k$. Since $W(n_\ell,m_\ell;\Psi_\ell)$ is null for $1\le \ell\le k$, $\prod_{\ell\ne i} W(n_\ell,m_\ell;\Psi_\ell)$ is also null. By Theorem \ref{t:product},
	\[\begin{split}
		&\fdim\bigg(\prod_{\ell=1}^{k}W(n_\ell,m_\ell;\Psi_\ell)\bigg)\\
		=&\min\bigg\{\fdim W(n_i,m_i;\Psi_i),\fdim\bigg(\prod_{\ell\ne i} W(n_\ell,m_\ell;\Psi_\ell)\bigg)\bigg\}\\
		=&\fdim W(n_i,m_i;\Psi_i).\qedhere
	\end{split}\]

\end{proof}

\subsection*{Acknowledgments}
Y. He was supported by NSFC (No. 12401108).
\subsection*{Data Availability}
 We do not analyse or generate any data sets, because our work proceeds within a theoretical and mathematical approach.
 \subsection*{Conflict of interest} On behalf of all authors, the corresponding author states that there is no Conflict of  interest.

\end{document}